\documentclass[12pt]{amsart}
\usepackage{amsrefs,amssymb,amsmath,amssymb,amsthm,verbatim}
\usepackage{graphicx}
\usepackage{pstricks}
\usepackage[ocgcolorlinks,linktoc=all]{hyperref}
\usepackage{lipsum}
\usepackage{esint}

\hypersetup{citecolor=blue,linkcolor=magenta}
\newtheorem{theorem}{Theorem}[section]
\newtheorem{claim}[theorem]{Claim}
\newtheorem{lemma}[theorem]{Lemma}
\newtheorem{proposition}[theorem]{Proposition}
\newtheorem{corollary}[theorem]{Corollary}
\theoremstyle{definition}
\newtheorem{definition}[theorem]{Definition}

\theoremstyle{remark}

\newtheorem{remark}[theorem]{Remark}
\numberwithin{equation}{section}

\newcommand{\dist}{\operatorname{dist}}
\renewcommand{\th}{\operatorname{Th}}

\newcommand{\spt}{\operatorname{spt}}
\newcommand{\diam}{\operatorname{diam}}

\newcommand{\graph}{\operatorname{graph}}
\newcommand{\tr}{\operatorname{tr}}

\newcommand{\eps}{\varepsilon}

\newcommand{\R}{\mathbb{R}}

\newcommand{\del}{\partial}

\newcommand{\mcfK}{\mathcal{K}}
\newcommand{\mcfM}{\mathcal{M}}

\newcommand{\vel}{\mathbf{vel}}
\newcommand{\Hvec}{\mathbf{H}}

\providecommand{\abs}[1]{\lvert #1\rvert}

\title[Mean convex flow with free boundary]{Mean convex mean curvature flow with free boundary}

\author[N. Edelen, R. Haslhofer, M. N. Ivaki, J. J. Zhu]{Nick Edelen, Robert Haslhofer, Mohammad N. Ivaki, Jonathan J. Zhu}

\begin{document}

\begin{abstract}
In this paper, we generalize White's regularity and structure theory for mean-convex mean curvature flow \cite{White_size,White_structure,White_subsequent} to the setting with free boundary. A major new challenge in the free boundary setting is to derive an a priori bound for the ratio between the norm of the second fundamental form and the mean curvature.  We establish such a bound via the maximum principle for a triple-approximation scheme, which combines ideas from Edelen \cite{Edelen16}, Haslhofer-Hershkovits \cite{HH18}, and Volkmann \cite{Vol}. Other important new ingredients are a Bernstein-type theorem and a sheeting theorem for low entropy free boundary flows in a halfslab, which allow us to rule out multiplicity 2 (half-)planes as possible tangent flows and, for mean convex domains, as possible limit flows. 
\end{abstract}

\maketitle

\tableofcontents

\section{Introduction}
A family of closed embedded hypersurfaces $\{M_t\}_{t\geq 0}$ in a Riemannian manifold $(N^{n+1},g)$ moves by mean curvature flow if the normal velocity at each point on the (hyper)surface is given by the mean curvature vector at that point. If the initial surface is mean-convex, i.e. if everywhere on the surface the mean curvature vector points inwards, then the flow can be continued uniquely through all singularities \cite{ES91,White_size}. This evolution can be described both in the language of partial differential equations as a level set flow \cite{ES91,CGG} and in the language of geometric measure theory as an integral Brakke flow \cite{brakke}.

White and Huisken-Sinestrari developed a far-reaching structure theory for the mean curvature flow of such mean-convex hypersurfaces. Their papers \cite{White_size,White_structure,White_subsequent,HS1,HS2} provide a package of estimates and structural results that yield a precise description of singularities and of high curvature regions (more recently, an alternative treatment of this theory has been given by Haslhofer-Kleiner \cite{HK}). In particular, their work implies that the parabolic Hausdorff dimension of the singular set is at most $n-1$, and all blowup limits are convex, noncollapsed, and smooth until they become extinct.\footnote{By Brendle \cite{Brendle_inscribed} and Haslhofer-Hershkovits \cite{HH18} this applies without dimensional restrictions and in general ambient manifolds.}

The goal of the present paper is to generalize this theory to hypersurface with free-boundary, i.e. hypersurfaces satisfying a Neumann-type boundary condition.  In analogy to the flow of closed hypersurfaces, the free-boundary flow through singularities can be described as a level-set flow \cite{GigaSato} or as a Brakke-type flow \cite{Edelen17}.  We prove here the existence of a unique mean-convex free boundary flow for all time and establish sharp results on the size and structure of singularities for this flow.  Recent work of White \cite{White_boundary} has developed an existence and regularity theory for flows with prescribed (Dirichlet) data, and in particular showed that the resulting flow is always regular near its Dirichlet data.  This is very different from the free-boundary setting, where one expects singularities on the boundary to be modelled on interior singularities.

\subsection{Mean convex free boundary flow}
To describe the setup, fix a smooth compact domain $D\subset N^{n+1}$. A smooth domain $K\subset D$ has \emph{free boundary} if $\partial K:=K\setminus \textrm{Int}_D(K)$ meets $\partial D$ orthogonally,\footnote{It is more convenient to phrase everything in terms of compact domains $K$ instead of hypersurfaces $M=\partial K$. These points of view are of course equivalent.} in other words
\begin{equation}
\langle N(x), \nu(x) \rangle= 0 \qquad \textrm{for all $x\in \partial K\cap \partial D$},
\end{equation}
where $\nu$ denotes the outward unit normal of $\partial K$ and $N$ denotes the outward unit normal to $\partial D$. These notions are illustrated in Figure \ref{figure_1}. We assume that $K$ is \emph{mean-convex}, i.e. that the mean curvature vector everywhere on $\partial K$ points inside $K$.

\begin{figure}[h]
\centering
\includegraphics[scale=1]{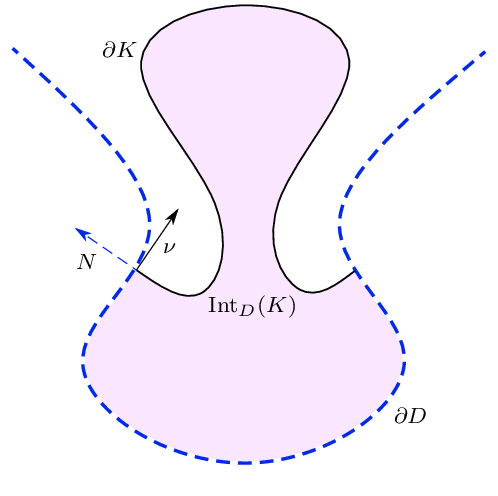}
\caption{Mean curvature flow with free boundary}\label{figure_1}
\end{figure}

Given $K$, there is a unique evolution by the \emph{free boundary level set flow} $F_t(K)$ from Giga-Sato \cite{GigaSato}. This has been described originally in terms of viscosity solutions of degenerate partial differential equations. Using the more geometric interpretation from Ilmanen \cite{I93}, $K_t:=F_t(K)\subset D$ simply is the maximal family of closed sets starting at $K$ that does not bump into any smooth free boundary subsolution (see Section \ref{Sec_level_set} for more detailed definitions). Mean convexity is preserved under the free boundary level set flow (see again Section \ref{Sec_level_set}), i.e.
\begin{equation}
K_{t_2}\subseteq \textrm{Int}_D K_{t_1} \quad \textrm{for all $t_2>t_1\geq 0$}.
\end{equation}
We denote by
\begin{equation}
\mathcal{K}:=\bigcup_{t\geq 0} K_t\times \{t\}\subset D\times \mathbb{R}_+
\end{equation}
the space-time track of the free boundary level set flow. Moreover, it can be checked (see Theorem \ref{thm_ell_reg} below) that
\begin{equation}
\mathcal{M}=\{ \mathcal{H}^n\lfloor\partial K_t \}_{t\geq 0}
\end{equation}
is a \emph{free boundary Brakke flow} as defined in \cite{Edelen17} (see also Section \ref{sec_free_brakke}).
The pair $(\mathcal{M},\mathcal{K})$ is called a \emph{mean-convex free boundary flow}.

\subsection{Size of the singular set}
Given a mean-convex free boundary flow $(\mathcal{M},\mathcal{K})$ in a domain $D\subset N^{n+1}$, its singular set $\mathcal{S}\subset D\times \mathbb{R}_+$ consists by definition of all points $X=(x,t)$ in the support of the flow, such that there is no two-sided space-time neighborhood of $X$ in which evolution can be described as smooth free boundary flow. The parabolic Hausdorff dimension of a subset of spacetime is the Hausdorff dimension with respect to the parabolic metric
\begin{equation}
d\left((x,t),(y,s)\right)=\max \{ |x-y| , |t-s|^{1/2}\}.
\end{equation}
For example, the parabolic Hausdorff dimension of $D\times \mathbb{R}_+$ is $n+3$. Our first main theorem shows that the singular set in any mean-convex free boundary flow is quite small:
\begin{theorem}[size of the singular set]\label{thm_size}
Let $(\mathcal{M},\mathcal{K})$ be a mean-convex free boundary flow in $D\subset N^{n+1}$. Then the parabolic Hausdorff dimension of its singular set $\mathcal{S}$ is
\begin{itemize}
\item at most $n-1$, if the domain $D$ is mean-convex, and
\item at most $n$, if the domain $D$ is arbitrary.
\end{itemize}
\end{theorem}
Simple examples show that the result is sharp. Indeed, rotationally symmetric thin shrinking tori respectively half-tori in a ball $D\subset\mathbb{R}^{n+1}$ show that an $(n-1)$-dimensional singular set can occur both in the interior and at the boundary. Likewise, the number $n$ is sharp for general barriers, since the surface can pop when it hits $\partial D$, c.f. \cite{Edelen17}. For example if $D=\bar{B}_4\setminus B_1\subset\mathbb{R}^{n+1}$ is an annulus, and $K=\bar{B}_{2}\setminus B_1$, then the shrinking sphere pops when it hits the inner boundary of the annulus and produces an $n$-dimensional singular set.

\subsection{Structure of singularities}
Let  $(\mathcal{M},\mathcal{K})$ be a mean-convex free boundary flow in a mean-convex domain $D$.

Given $X_i=(x_i,t_i)\to X=(x,t)$ in the support of the flow and $\lambda_i\to \infty$, we consider the \emph{blowup sequence} $(\mathcal{M}^i,\mathcal{K}^i)$ which is obtained from $(\mathcal{M},\mathcal{K})$ by translating $X_i$ to the origin and parabolically rescaling by $\lambda_i$. It is always possible to find a convergent subsequence, where $\mathcal{M}^i$ converges in the sense of (free boundary) Brakke flows \cite{I94,Edelen17} and $\mathcal{K}^i$ converges in the Hausdorff sense. Any subsequential limit $(\mathcal{M}',\mathcal{K}')$ is called a \emph{limit flow at $X$}. Here, $(\mathcal{M}',\mathcal{K}')$ is either a flow without boundary in $\mathbb{R}^{n+1}$ or a flow with free boundary defined in a halfspace $\mathbb{H}\subset\mathbb{R}^{n+1}$. If $X_i=X$ we call it a \emph{tangent flow at $X$}. Tangent flows are always backwardly selfsimilar \cite{Huisken_monotonicity,Edelen17}.

\begin{theorem}[structure of singularities]\label{thm_structure}
Let $(\mathcal{M},\mathcal{K})$ be a mean-convex free boundary flow in a mean-convex domain $D\subset N^{n+1}$. Let $(\mathcal{M}',\mathcal{K}')$ be a limit flow, and denote by $T_{\textrm{ext}}=T_{\textrm{ext}}(\mathcal{M}',\mathcal{K}')\in (-\infty,+\infty]$ its extinction time, i.e. the supremum of all $t$ with $K'_t\neq\emptyset$. Then:
\begin{itemize}
\item $\mathcal{M}'$ has multiplicity one and its support equals $\partial\mathcal{K}'$.
\item $(\mathcal{M}',\mathcal{K}')\cap \{t <T_{\textrm{ext}}\}$ is smooth.
\item $\mathcal{K}'$ has convex time slices, i.e. $K_t'$ is convex for every $t$.
\item $\mathcal{K}'$ is either a static halfspace or quarterspace, or it has strictly positive mean curvature and sweeps out the whole space or halfspace, i.e. $\cup_t \partial K_t' = \mathbb{R}^{n+1}$ or  $\cup_t \partial K_t' = \mathbb{H}$, respectively.
\item $\mathcal{K}'$ is $1$-noncollapsed, i.e. every $p\in\partial K_t'$ admits interior and exterior balls of radius $1/H(p,t)$.
\end{itemize}
Furthermore, if $(\mathcal{M}',\mathcal{K}')$ is backwardly selfsimilar, then it either (i) a static multiplicity one plane or halfplane, (ii) a round shrinking sphere or halfsphere, or (iii) a round shrinking cylinder or halfcylinder.
\end{theorem}
Examples where halfcylinders occur as tangent flows are neckpinches at the boundary. Note that there are two types of halfcylinders with free boundary, depending on whether the axis is contained in $\partial\mathbb{H}$ or perpendicular to it. An illustrative example where a limit flow is not backwardly selfsimilar is a degenerate neck pinch at the boundary, in which case some limit flow $(\mathcal{M}',\mathcal{K}')$ would be a halfbowl. These examples are illustrated in Figure \ref{figure_neckpinches}.

\begin{figure}[h]
\centering
\includegraphics[scale=0.4]{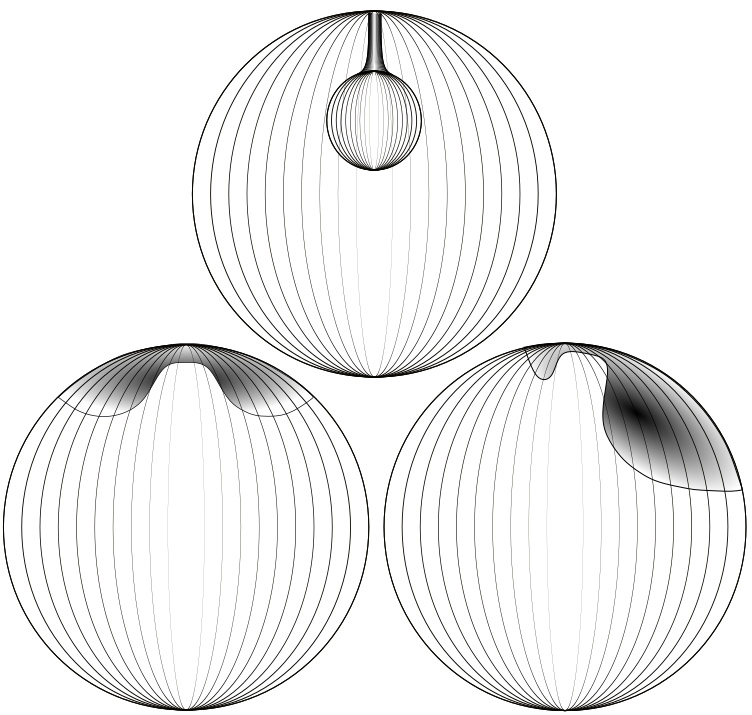}
\caption{Neckpinches at the boundary}\label{figure_neckpinches}
\end{figure}

Together with the recent breakthroughs by Brendle-Choi \cite{BC} and Angenent-Daskalopoulos-Sesum \cite{ADS} we obtain for $n=2$:

\begin{corollary}[classification of limit flows]
Let $(\mathcal{M},\mathcal{K})$ be a two-dimensional mean-convex free boundary flow in a mean-convex domain $D\subset N^{3}$. Then any limit flow $(\mathcal{M}',\mathcal{K}')$ is one of the following:
\begin{itemize}
\item a static multiplicity one plane or halfplane,
\item a round shrinking sphere or halfsphere,
\item a round shrinking cylinder or halfcylinder,
\item a translating bowl or halfbowl,
\item an ancient oval or halfoval.
\end{itemize}
\end{corollary}

In any dimension, Theorem \ref{thm_structure} (structure of singularities), in particular implies the following estimates for the flow $(\mathcal{M},\mathcal{K})$ itself.

\begin{corollary}[estimates for mean-convex free boundary flow]
Let $(\mathcal{M},\mathcal{K})$ be a mean-convex free boundary flow in a mean-convex domain $D\subset N^{n+1}$. Then
\begin{itemize}
\item (sharp noncollapsing) For any $\alpha<1$ and $T<\infty$ there exists an $H_0<\infty$ such that for any $t\leq T$ any $p\in \partial K_t$ with $H(p,t)\geq H_0$ admits interior and exterior balls of radius at least $\alpha/H(p,t)$.
\item (convexity estimate) For all $\eps>0$ and $T<\infty$ there exists an $H_0<\infty$ such that for any $t\leq T$ any $p\in \partial K_t$ with $H(p,t)\geq H_0$ satisfies $\lambda_1\geq -\eps H$.
\item (gradient estimate) For all $T<\infty$ there is a $C<\infty$ such that at all times $t\leq T$ at all points we have $|\nabla H|\leq CH^2$.
\end{itemize}
\end{corollary}

We mention that a related convexity estimate was proved in \cite{Edelen16}. On the other hand, the gradient estimate for mean-convex free boundary flows only follows after establishing our main structure theorem (Theorem \ref{thm_structure}), and we are not aware of any shorter path towards establishing such a gradient estimate directly. This is similar in spirit to the fact that Perelman obtained the gradient estimate for 3d Ricci flow as a corollary after establishing his canonical neighborhood theorem \cite{Per1,Per2}.

\subsection{Long-time behavior} Mean curvature flow of closed surfaces in Euclidean space always becomes extinct in finite time, since there are no closed minimal surfaces it can converge to. The mean curvature flow with free boundary can have more interesting long-time behavior even for $D\subset\mathbb{R}^{n+1}$:

\begin{theorem}[long-time behavior]\label{thm_longtime}
Let $(\mathcal{M},\mathcal{K})$ be a mean-convex free boundary flow in a mean-convex domain $D\subset N^{n+1}$, and set
\begin{equation}
K_\infty :=\bigcap_t K_t.
\end{equation}
Then either $K_\infty=\emptyset$ or $K_\infty$ has finitely many connected components. The boundary of each component is a stable free boundary minimal surface whose singular set has Hausdorff dimension at most $n-7$. Furthermore, $\partial K_t$ converges smoothly (either locally one-sheeted or two-sheeted depending whether or not the component has interior points) to $\partial K_\infty$ away from the singular set of $\partial K_\infty$.
\end{theorem}

For example, consider the case where $D$ looks like a dumbbell, which contains a strictly stable free boundary minimal disc $\Sigma$. Fix $\eps>0$ small enough.
If $K=\{x\in D: d(x,\Sigma)\leq \eps\}$, then we get two-sheeted convergence to $\Sigma$. On the other hand, letting $C$ be one of the components of $D\setminus \Sigma$ and setting
$K=\{x\in D: d(x,C)\leq \eps\}$, we get one-sheeted convergence.

\subsection{Some remarks on the proofs}
As the reader will expect, one of the key steps is of course to rule out blowup limits of higher multiplicity. In the case of mean-convex surfaces without boundary there is a short maximum principle argument by Andrews \cite{Andrews12}. Unfortunately, we have not been able to generalize the argument by Andrews to the free boundary setting (related to this, Brendle's proof of the Lawson conjecture \cite{Brendle_lawson} does not seem to generalize in any obvious way to establish uniqueness of the critical catenoid). We thus follow White's original approach from \cite{White_size}. To this end, we prove free boundary versions of the expanding hole lemma, the Bernstein-type theorem for low entropy flows in a slab, and the sheeting theorem, and use them to rule out static and quasistatic planes or free boundary halfplanes with density two as potential tangent flows or limit flows (see Section \ref{sec_mult_one}).

Another major issue is to rule out nontrivial minimal cones, such as the Simons cone, as potential blow up limits. This requires an a priori bound for the ratio between the norm of the second fundamental form and the mean curvature, which turns out to be substantially more involved than in the case without boundary. The technical heart of the present paper is the following:

\begin{theorem}[elliptic regularization and consequences]\label{thm_ell_reg}
Let $K_t:=F_t(K)$ be a free boundary level set flow with smooth strictly mean-convex initial data $K\subset D$. Then:
\begin{itemize}
\item $\{K_t\times \mathbb{R}\}_{t> 0}$ arises as a limit of smooth free boundary flows.
\item $t\mapsto \mathcal{H}^n\lfloor\partial K_t$ is a free boundary Brakke flow.\footnote{In the proof we will derive a Lipschitz estimate, which together with the co-area formula implies that the reduced boundary $\partial^\ast\! K_t$ agrees with the topological boundary $\partial K_t$ up to $\mathcal{H}^n$-measure zero at almost every time, c.f. \cite{MetzgerSchulze}.}
\item $K_t$ is one sided minimizing, i.e. $\mathcal{H}^n(\partial K_t)\leq \mathcal{H}^n(\partial F)$ whenever $K_t\subseteq F\subseteq K$.
\item There exist constants $c=c(K,D)>0$ and $\rho=\rho(K,D)<\infty$ such that for all $t>0$ we have
\begin{equation*}
\inf_{\partial K_t^\textrm{reg}}H\geq c e^{-\rho t}.
\end{equation*}
\item For every $\delta>0$ there exist constants $C=C(K,D,\delta)<\infty$ and $\rho=\rho(K,D,\delta)<\infty$ such that for every $t\geq \delta$ we have
\begin{equation*}
\sup_{ \partial K_t^\textrm{reg}}\frac{|A|}{H}\leq C e^{\rho t}.
\end{equation*}
\end{itemize}
\end{theorem}

To prove Theorem \ref{thm_ell_reg} (elliptic regularization and consequences) we use a triple-approximate elliptic regularization scheme. We use a capillary parameter to deal with the lack of $L^\infty$-estimate in case the flow does not become extinct in finite time (c.f. Haslhofer-Hershkovits \cite{HH18}). Next, we bend the corners of $K$ to deal with the mixed Dirichlet-Neumann problem (c.f. Volkmann \cite{Vol}). Moreover, we have to perturb the second fundamental form to get a reasonable normal derivative at the barrier (c.f. Edelen \cite{Edelen16}). Finally, we need an additional weight function to force the extrema away from the barrier (c.f. Edelen \cite{Edelen16}). We carry this out in Section \ref{sec_ell_reg}.

Once Theorem \ref{thm_ell_reg} (free boundary level set flow) and multiplicity one (see Section \ref{sec_mult_one}) are established, all the remaining steps proceed broadly following the proof strategy of \cite{White_size,White_structure}, the main point being to rule out density 2 (planar) limit flows. In addition to multiplicity 2 planes and free boundary half-planes, there is the new possibility of a (quasi)-static limit flow defined in a half-space and coincident with the barrier plane, which is not directly analogous to the case without boundary. We rule out this potential scenario if the barrier is mean-convex.\footnote{This assumption on the barrier is indeed necessary. For general barriers, the flow can ``pop", which is captured by quasistatic density two tangent flows.}

\bigskip

\noindent\textbf{Acknowledgments.} NE has been supported in part by an NSF postdoctoral fellowship (DMS-1606492). RH has been partially supported by an NSERC Discovery Grant (RGPIN-2016-04331) and a Sloan Research Fellowship. MI has been supported by a Jerrold E. Marsden postdoctoral fellowship from the Fields Institute. JZ has been supported in part by an NSF postdoctoral fellowship (DMS-1802984).

\bigskip

\section{Notation and conventions}
Let us fix a Riemannian manifold $(N^{n+1},g)$, and a compact domain $D \subset (N^{n+1},g)$ with smooth boundary $\del D$.
For any domain $K\subset D$ we write $\mathring{K}$ for the interior of $K$ viewed as a subset of $N$, and $\textrm{Int}_DK$ for the interior of $K$ viewed as a subset of the topological space $D$. We call $\partial K:= K\setminus \textrm{Int}_DK$ the Dirichlet boundary of $K$, and $\delta K:=K\setminus (\mathring{K}\cup \partial K)$ the Neumann boundary of $K$.
We let $\nu$ be the outward unit normal of $\partial K$ and use the sign convention that $\mathbf{H} = -H \nu$. We write $N$ for the outward unit normal to $\partial D$.

\bigskip

\section{Free boundary level set flow}\label{Sec_level_set}
A smooth family of closed domains $\{L_t\subset D\}_{t\in [a,b]}$ is a \emph{smooth free boundary subsolution} if $\del L_t$ moves inwards at least as fast as the mean curvature flow, and hits $\del D$ in a convex fashion, i.e.
\begin{align}\label{subsup def}
\langle \vel,  \nu \rangle \leq \langle \Hvec , \nu\rangle, \quad N \cdot \nu \geq 0,
\end{align}
where $\mathbf{vel}$ denotes the normal velocity of $\del L_t$. A family of closed sets $\{K_t\subset D\}_{t\in [a,b]}$ is a \emph{set theoretic subsolution of the free boundary mean curvature flow} if it avoids any smooth free boundary subsolution, i.e. whenever $\{ L_t\subset  D \}_{t\in [t_0,t_1]}$, where $a\leq t_0\leq t_1\leq b$, is a smooth free boundary subsolution with $K_{t_0}\cap L_{t_0}=\emptyset$, then $K_{t}\cap L_{t}=\emptyset$ for all $t\in [t_0,t_1]$. The \emph{free boundary level set flow} $F_t(K)$ is the maximal set theoretic subsolution with initial condition $F_0(K)=K$. By the work of Giga-Sato \cite{GigaSato} the free boundary level set flow has the following basic properties:
\begin{enumerate}
\item Consistency with smooth flows: If $\{M_t\}_{t\in[0,T]}$ is a smooth free-boundary mean curvature flow, then $M_t=F_t(M_0)$.
\item Avoidance: If $K$ and $L$ are disjoint, then $F_t(K)$ and $F_t(L)$ are also disjoint and $F_t(K\cup L)=F_t(K)\cup F_t(L)$.
\item Inclusion: If $K\subseteq L$, then $F_t(K)\subseteq F_t(L)$.
\item Semigroup property: $F_{t_1+t_2}(K)=F_{t_2}(F_{t_1}(K))$.
\item Strict inclusion: If $K\subseteq \operatorname{Int}_D L$, then $F_t(K)\subseteq \operatorname{Int}_DF_t(L)$.
\end{enumerate}

A closed subset $K\subset D$ is said to be \emph{mean-convex} if for some $\delta>0$,
 $ F_t(K)\subseteq  \operatorname{Int}_{D} K$ for all $0<t<\delta$. If $K$ is smooth this is equivalent to the condition that $H\geq 0$ everywhere and not identically zero.

\begin{proposition}
Let $K\subset D$ be mean-convex. Mean convexity is preserved under the free boundary level set flow, i.e. $F_{t_2}(K)\subseteq \textrm{Int}_D F_{t_1}(K)$ for all $t_2>t_1\geq 0$. Moreover, the boundaries $\partial F_t(K)$ form a possibly singular foliation of $K\setminus \cap_{t=0}^\infty F_t(K)$, and it holds that $F_t(\partial K)=\partial F_t(K)$. In particular, the flow $t\mapsto F_t(\partial K)$ is nonfattening.
\end{proposition}

\begin{proof}
This follows from the above basic properties, arguing as in \cite{White_size}*{Sec. 3}. Namely, using the definition of mean-convex and the basic properties (4) and (5) we infer that
\begin{equation}\label{eq_mean_conv_ind}
F_{t+h}(K)\subseteq \operatorname{Int}_DF_t(K)
\end{equation}
for all $t\geq 0$ and all $0<h<\delta$. By induction we conclude that \eqref{eq_mean_conv_ind} holds for all $t\geq 0$ and all $0<h<j\delta$, where $j=1,2,\ldots$, hence mean convexity is preserved. In particular, the sets $\{\partial F_t(K)\}_{t\geq 0}$ are disjoint. Next, given any $x\in K\setminus \cap_{t=0}^\infty F_t(K)$, denoting by $u(x)$ the last time such that $x\in F_{u(x)}(K)$, we have $x\in\partial F_{u(x)}(K)$. Hence
\begin{equation}
K\setminus \cap_{t=0}^\infty F_t(K)=\cup_{t\geq 0} \partial F_t(K).
\end{equation}

Next, it follows from (2) and (5) that $t\mapsto \partial F_t(K)$ is a set-theoretic subsolution. Also, note that $F_t(\partial K)\subseteq F_t(K)$ be property (3). Together with the observation that $F_t(\partial K)$ is disjoint from
\begin{equation}
\bigcup_{h>0}F_{t+h}(K)=\operatorname{Int}_D F_t(K)
\end{equation}
by (5), we infer that $F_t(\partial K)\subseteq \partial F_t(K)$. Recalling that $t\mapsto \partial F_t(K)$ is a set-theoretic subsolution with initial condition $\partial K$, while $F_t(\partial K)$ is the maximal set-theoretic subsolution with initial condition $\partial K$, we conclude that $F_t(\partial K)=\partial F_t(K)$. Since $\textrm{Int}_D\partial F_t(K)=\emptyset$, this finishes the proof of the proposition.
\end{proof}

\bigskip

\section{Elliptic regularization and consequences}\label{sec_ell_reg}
In this section, we prove Theorem \ref{thm_ell_reg} (elliptic regularization and consequences).
\subsection{Triple approximation scheme}\label{Triple approximation}
The aim of this subsection is to construct solutions of a mixed Dirichlet-Neumann problem:

\begin{theorem}\label{cor tau=0}
Given a strictly mean-convex domain $K\subset D$, and constants $\eps,\sigma>0$, there exists a unique solution $u_{\eps,\sigma}\in  C^{\infty}(K\backslash\partial K)\cap C^{0,1}(K)$ of the problem
\begin{align}\label{tau=0}
\operatorname{div}\left(\frac{Du_{\eps,\sigma}}{\sqrt{\eps^2+\abs{Du_{\eps,\sigma}}^2}}\right)+\frac{1}{\sqrt{\eps^2+\abs{Du_{\eps,\sigma}}^2}}&=\sigma u_{\eps,\sigma}  &  \textrm{in} \,\,\mathring{K}\nonumber \\
\langle N, Du_{\eps,\sigma}\rangle&= 0 & \textrm{on} \,\, \delta K\\
u_{\eps,\sigma}&= 0 & \textrm{on} \,\,\partial K.\nonumber
\end{align}
\end{theorem}

To solve the problem \eqref{tau=0}, we need some further approximations.
Namely, to deal with the mixed Dirichlet-Neumann boundary we approximate the initial domain $K$ by domains $K^\tau$ ($\tau>0$) as in Volkmann \cite{Vol}*{p. 74}, which in particular satisfy the strict angle condition
\begin{equation}\label{eq_strict_angle}
\langle\nu_{\partial K^{\tau}}, N\rangle\leq -\tau/2.
\end{equation}

Given $\eps,\sigma,\tau>0$, we then consider the triple approximation problem
\begin{align}\label{trip_approx_prob}
\operatorname{div}\left(\frac{Du_{\eps,\sigma,\tau}}{\sqrt{\eps^2+\abs{Du_{\eps,\sigma,\tau}}^2}}\right)
+\frac{1}{\sqrt{\eps^2+\abs{Du_{\eps,\sigma,\tau}}^2}}&=\sigma u_{\eps,\sigma,\tau}  &  \textrm{in} \,\,\mathring{K}^\tau\nonumber\\
\langle N, Du_{\eps,\sigma,\tau}\rangle&= 0 & \textrm{on} \,\, \delta K^\tau\\
u_{\eps,\sigma,\tau}&= 0 & \textrm{on} \,\,\partial K^\tau.\nonumber
\end{align}

To solve \eqref{trip_approx_prob} we use the continuity method, i.e. we introduce yet another parameter $\kappa\in [0,1]$ and consider the problem
\begin{align}\label{cont_meth}
\operatorname{div}\left(\frac{Du_{\eps,\sigma,\tau,\kappa}}{\sqrt{\eps^2+\abs{Du_{\eps,\sigma,\tau,\kappa}}^2}}\right)
+\frac{\kappa}{\sqrt{\eps^2+\abs{Du_{\eps,\sigma,\tau,\kappa}}^2}}&=\sigma u_{\eps,\sigma,\tau,\kappa}  &  \textrm{in} \,\,\mathring{K}^\tau\nonumber\\
\langle N,Du_{\eps,\sigma,\tau,\kappa}\rangle&= 0 & \textrm{on} \,\, \delta K^\tau\\
u_{\eps,\sigma,\tau,\kappa}&= 0 & \textrm{on} \,\,\partial K^\tau.\nonumber
\end{align}
For $\kappa=0$ the problem \eqref{cont_meth} has the obvious solution $u_{\eps,\sigma,\tau,0}=0$. We will now derive the needed a-priori estimates for $\kappa\in(0,1]$. Note first that we have the sup-bound
\begin{equation}\label{easy_sup_bound}
0 \leq u_{\eps,\sigma,\tau,\kappa}\leq \frac{\kappa}{\sigma\eps},
\end{equation}
which follows directly from the maximum principle. To proceed further, we consider the graph
\begin{equation}
M^{\eps,\sigma,\tau,\kappa}=\graph(u_{\eps,\sigma,\tau,\kappa}/\eps)\subset N\times \mathbb{R}_+.
\end{equation}
We write $\eta=\tfrac{\partial}{\partial z}$ for the unit vector in $\mathbb{R}_+$ direction, and $\nu$ for the upward pointing unit normal of $M$ (here and in the following we drop the dependence on $(\eps,\sigma,\tau,\kappa)$ in the notation when there is no risk of confusion). Written more geometrically, problem  \eqref{cont_meth} takes the form
\begin{align}\label{equation_geom}
&\textrm{$M$ satisfies } H+\sigma u=\kappa V,\nonumber\\
&\textrm{with free boundary on}~\partial D\times \mathbb{R}_+,\\
&\textrm{and Dirichlet boundary on}~\partial K^\tau \times \{0\},\nonumber
\end{align}
where $H$ is the mean curvature of $M\subset N\times \mathbb{R}_+$, and
\begin{equation}
V=\tfrac{1}{\eps}\langle \eta,\nu\rangle=\frac{1}{\sqrt{\eps^2+\abs{Du}^2}}.
\end{equation}
We recall from \cite{HH18}*{Lem. 2.5} that
\begin{equation}\label{prod_laplacian}
\Delta V=\tfrac{1}{\eps}\langle  \eta, \nabla H \rangle-\left(|A|^2+\operatorname{Rc}(\nu,\nu)\right)V.
\end{equation}
Let $d: D\to \R_{+}$ be a $C^2$-function satisfying
\begin{equation}
d|_{\partial D}\equiv0,\quad N(d)=-1,\quad |Dd|\leq 1,\quad |D^2d|\leq 10C_{\partial D},
\end{equation}
where $C_{\partial D}$ denotes the maximal curvature of the barrier. We will sometimes tacitly view $d$ as a function on $M$ which is independent of the $z$-coordinate.

\begin{lemma}\label{weigh_fn}
The weight function $w=e^{mz-bd}$, where $m$ and $b$ are constants, satisfies
\begin{align}
\nabla \log w=& m\eta^\top-b\nabla d,
\end{align}
and
\begin{align}\label{lemma_lapl_weight}
\Delta w= \left( |m\eta^\top - b\nabla d|^2-b\operatorname{tr}_{TM}D^2d-(m\eps V-b \nu(d)) H\right)w.
\end{align}
\end{lemma}

\begin{proof}
The first formula is immediate. To prove the second formula choosing an orthonormal frame $\{e_i\}$ with $\nabla_{e_i}e_{j}=0$ at the point in consideration we compute
\begin{align}
\Delta w &= \nabla_{e_i}\left(  \langle m\eta-bD d,e_i\rangle w\right)\\
&=\langle m\eta-bD d,e_i\rangle\langle m\eta-bD d,e_i\rangle w\\
&\qquad-b \langle D^2 d, e_i\otimes e_i\rangle w -H\langle m\eta -bD d,\nu\rangle w.\nonumber
\end{align}
This implies the assertion.
\end{proof}

\begin{lemma}\label{basic equations} The function $Vw:M^{\eps,\sigma,\tau,\kappa}\to \mathbb{R}_+$ satisfies
\begin{align}\label{normder_Vw}
N\left(\log(Vw)\right)=  b +  A_{\partial D}(\nu,\nu),
\end{align}
and
\begin{align}
\Delta(Vw)=&2\langle \nabla \log w,\nabla (Vw)\rangle+\tfrac{1}{\varepsilon}\langle  \eta, \kappa \nabla V-\sigma \nabla u \rangle w\nonumber\\
&-\left(|m\eta^\top-b\nabla d|^2+|A|^2+\operatorname{Rc}(\nu,\nu)+b\operatorname{tr}_{TM}D^2d\right)Vw\\
&-(\kappa V-\sigma u)(m\eps V-b\nu(d)) Vw.\nonumber
\end{align}
\end{lemma}

\begin{proof}
Differentiating $\langle\nu,N\rangle=0$ and using \eqref{equation_geom} one obtains that
\begin{equation}
N(V)=A_{\partial D}(\nu,\nu) V.
\end{equation}
Together with $N(\log w)=b$, this yields \eqref{normder_Vw}. Using that
\begin{align*}
\Delta(Vw)=&w\Delta V+V\Delta w+2\langle \nabla \log w,\nabla (Vw)\rangle-2Vw|\nabla\log w|^2,
\end{align*}
the second formula follows from \eqref{prod_laplacian} and \eqref{lemma_lapl_weight}.
\end{proof}

\begin{proposition}\label{V_bd_1}
Let $b=2C_{\partial D}$ and $m=\max(20C_{\partial D},2\max\limits_{D}|\operatorname{Rc}|^{\frac{1}{2}})$. Then the function $V:M^{\eps,\sigma,\tau,\kappa}\to \mathbb{R}$ satisfies
\begin{equation}
V(x,z) \geq e^{-b\max d} \min  \left(\frac{1}{2\varepsilon},\min_{\partial K^\tau}V\right)  e^{-mz}.
\end{equation}
\end{proposition}

\begin{proof}
By equation \eqref{normder_Vw}, the normal derivative of $Vw$ is positive at the barrier for $b=2C_{\partial D}$. Therefore, the minimum of $Vw$ is attained in $K^\tau\setminus\delta K^\tau$. If the minimum is attained on $\partial K^\tau$ or if the minimum is at least $\frac{1}{2\eps}e^{-b\max d}$, then we are done. Suppose now towards a contradiction that the minimum of $Vw$ is attained in $\mathring{K}^\tau$ and is less than $\frac{1}{2\eps}e^{-b\max d}$.
By Lemma \ref{weigh_fn} and Lemma \ref{basic equations} at such an interior point we get
\begin{align}
 \nabla \log V=b\nabla d- m\eta^{\top},
\end{align}
and
\begin{multline}\label{dropped_good_term}
0 \geq |m\eta^\top-b\nabla d|^2+|A|^2+\operatorname{Rc}(\nu,\nu)+b\operatorname{tr}_{TM}D^2d\\
+(\kappa V-\sigma u)\left(m\eps V-b\nu(d)\right)+\tfrac{\kappa}{\varepsilon}\langle m\eta^{\top}-b\nabla d,\eta\rangle,
\end{multline}
where we also used that $\langle \nabla u,\eta\rangle=\eps |\eta^{\top}|^2 \geq 0$.
Furthermore, taking also into account the graphical identity
\begin{align}\label{eq_graphical_id}
\langle\nabla d,\eta\rangle=-\eps V \nu(d),
\end{align}
this implies
\begin{multline}\label{x}
0\geq m^2|\eta^\top|^2+m\left( \tfrac{\kappa}{\varepsilon}|\eta^{\top}|^2-\sigma u\eps V\right)\\
+\operatorname{Rc}(\nu,\nu)-b\left(2m\varepsilon V+\sigma u+10C_{\partial D}\right).
\end{multline}
Since $\varepsilon V<\frac{1}{2}$, we have $|\eta^{\top}|^2\geq \frac{3}{4}$. Hence, $m\geq 2\max\limits_{D}|\operatorname{Rc}|^{\frac{1}{2}}$ yields
\begin{align}
0\geq \frac{m^2}{2}+\frac{m\kappa}{4\varepsilon}-2C_{\partial D}\left(m+\frac{\kappa}{\varepsilon}+10C_{\partial D}\right),
\end{align}
but this contradicts $m\geq 20C_{\partial D}$.
\end{proof}

Note that a lower bound for $V$ is equivalent to an upper bound for $\abs{D u}$. The next lemma provides a uniform lower bound for $\min_{\partial K^\tau}V$.

\begin{lemma}\label{grad bound with tau}
There exists a constant $C=C(\eps,\sigma,K)<\infty$, such that
\begin{equation}
\sup_{\partial K^\tau}|D u_{\eps,\sigma,\tau,\kappa}| \leq C.
\end{equation}
\end{lemma}
\begin{proof}
We argue as in the proof of \cite{Vol}*{Lem. 3.9} and seek a super-solution of the form
\begin{equation}
v=\alpha r\cdot g(s),
\end{equation}
where $r$ is the distance from $\partial K^{\tau}$ and $s$ is the distance from $\partial D$. Here, $g(s)=1+f(\beta  s)$, with $f(t)$ denoting a smooth mollification of $\max(0,1-t)$.
We work in the region $T_{\rho}:=\{x\in K^{\tau}: r(x)<\rho\}$, where $\rho>0$ is in particular small enough to ensure that $r$ is smooth on $T_\rho$.
Due to the sup-bound (\ref{easy_sup_bound}) we have  $v\geq u_{\eps,\sigma,\tau,\kappa}$ on $\partial T_\rho$, provided
\begin{equation}
\alpha\geq \frac{1}{\sigma\varepsilon\rho}.
\end{equation}
Let $\gamma=\gamma(D,K)>0$ be a constant such that for any $\tau\in[0,1]$ we have
$N(\log \textrm{dist}(\cdot, \partial K^\tau))\geq -\gamma$ on $\delta K^\tau$. Then the normal derivative of $v$ on the Neumann-boundary $\delta T_{\rho}$ satisfies
\begin{align}
N(v)=\alpha g(s) N(r)+\alpha r g'(s)N(s)> 0,
\end{align}
provided
\begin{equation}
\beta > 4\gamma.
\end{equation}
By the maximum principle, it is thus enough to show that
\begin{align}\label{barrier_est_toshow}
\mathcal{D}v:=&\operatorname{div}\left(\frac{Dv}{\sqrt{\varepsilon^2+|Dv|^2}}\right)+\frac{\kappa}{\sqrt{\varepsilon^2+|Dv|^2}}-\sigma v\leq 0
\end{align}
in $\mathring{T}_\rho$ (once this is done, one concludes that $u_{\eps,\sigma,\tau,\kappa}\leq v$ in $T_\rho$, which yields the assertion of the lemma).

Let us now show that \eqref{barrier_est_toshow} indeed holds for suitable choice of constants $\alpha,\beta,\rho$. To this end, we start by estimating
\begin{align}
\mathcal{D}v\leq& \frac{1}{\sqrt{\varepsilon^2+|Dv|^2}}\left(\Delta v+1-\frac{1}{2(\varepsilon^2+|Dv|^2)}\langle D|Dv|^2,Dv\rangle\right).
\end{align}

If $\rho$ is small enough, then by the Riccati equation we have
\begin{equation}
\Delta r\leq -\tfrac{1}{2}\min_{\partial K} H.
\end{equation}
Thus, we obtain
\begin{align}\label{x1}
\Delta v=&\alpha g\Delta r+ 2\alpha\langle Dr,Dg \rangle +  \alpha r\Delta g\nonumber\\
\leq&-\frac{\alpha}{2}\min_{\partial K} H+2\alpha\beta f'\langle Dr,Ds\rangle+C(\beta)\alpha\rho.
\end{align}
Next, we calculate
\begin{align}
\frac{1}{\alpha}Dv=gDr+\beta rf'Ds,
\end{align}
and
\begin{align}\label{eq_dv_sq}
\frac{1}{\alpha^2}|Dv|^2=g^2+\beta ^2r^2f'^2+2\beta rf'g\langle Dr,Ds\rangle.
\end{align}
In particular, we see that
\begin{equation}
|Dv|^2\geq \frac{\alpha^2}{2}
\end{equation}
for $\rho$ small enough. Taking another derivative of \eqref{eq_dv_sq} we get
\begin{align}
\frac{1}{\alpha^2}D|Dv|^2=2\beta f'gDs+2\beta f'g\langle Dr,Ds\rangle Dr+\mathcal{R},
\end{align}
where the remainder satisfies $|\mathcal{R}|\leq C(\beta)\rho$. This yields
\begin{align}\label{eq_compl_rem}
\left| \frac{1}{\alpha^3}\langle D|Dv|^2,Dv\rangle-4\beta f'g^2\langle Dr,Ds\rangle\right| \leq C(\beta)\rho.
\end{align}
Finally, using again \eqref{eq_dv_sq} we see that
\begin{equation}
2\alpha\beta f' \langle Dr,Ds\rangle - \frac{4\alpha^3 \beta  f'g^2 \langle Dr,Ds\rangle}{2(\eps^2+|Dv|^2)}\\
\leq \frac{\beta}{\alpha} + C(\beta)\alpha\rho.
\end{equation}
Putting everything together we conclude that
\begin{equation}
{\sqrt{\varepsilon^2+|Dv|^2}}\mathcal{D}v\leq -\frac{\alpha}{2}\min_{\partial K} H+1+\frac{\beta}{\alpha}+C(\beta)\alpha\rho <0,
\end{equation}
provided we first fix $\beta$ large enough, and then choose $\alpha$ very large and set $\rho=\frac{1}{\sigma\eps\alpha}$. This proves the lemma.
\end{proof}

We can now prove existence for our triple approximate problem.

\begin{theorem}\label{existence}
There exists $u_{\eps,\sigma,\tau}\in C^{\infty}(K^\tau\setminus\partial K^\tau)\cap C^{1,\alpha}(K^\tau)$, where $\alpha=\alpha(\tau)>0$, which solves the problem \eqref{trip_approx_prob}.
\end{theorem}

\begin{proof}
As in \cite{Lieb1,Mar,Vol} we work with the weighted H\"{o}lder space
\begin{equation}
H_{2,\alpha}^{(-1-\alpha)}(K^\tau):=\{u:\|u\|_{2,\alpha}^{(-1-\alpha)}<\infty\},
\end{equation}
equipped with the norm
\begin{align}
\|u\|_{2,\alpha}^{(-1-\alpha)}:=\sup_{\delta>0}\delta \|u\|_{C^{2,\alpha}(\textrm{Int}_\delta K^\tau)},
\end{align}
where
\begin{equation}
\textrm{Int}_\delta K^\tau:=\{ x\in K^\tau : d(x,\partial K^\tau)\geq \delta\}.
\end{equation}
It follows directly from the definitions that
\begin{equation}
H_{2,\alpha}^{(-1-\alpha)}(K^{\tau})\subseteq C^{2,\alpha}_{\textrm{loc}}(\textrm{Int}_D{K}^\tau)\cap C^{1,\alpha}(K^\tau).
\end{equation}

Fix $\eps,\sigma,\tau>0$, and consider
\begin{equation}
I:=\{\kappa\in [0,1]\,|\, \textrm{\eqref{cont_meth} has a solution in}~H_{2,\alpha}^{(-1-\alpha)}(K^{\tau})\}.
\end{equation}
We want to show that $1\in I$, provided $\alpha=\alpha(\tau)>0$ is sufficiently small. Since $0\in I$, it suffices to show that $I\subseteq [0,1]$ is open and closed.\\

Note that by equation \eqref{easy_sup_bound}, Proposition \ref{V_bd_1} and Lemma \ref{grad bound with tau} we have the a priori estimate
\begin{equation}
\sup_{K^\tau}\left( u+|Du|\right)\leq C,
\end{equation}
where $C<\infty$ is independent of $\kappa$. Since by \eqref{eq_strict_angle} the corners of the domain $K^\tau$ have angles strictly less than $\pi/2$, for $\alpha=\alpha(\tau)>0$ small enough, we can now apply Lieberman's estimates for mixed boundary value problems \cite{Lieb1,Lieb2}, to get the a priori estimate
\begin{equation}
\|u\|_{2,\alpha}^{(-1-\alpha)}\leq C,
\end{equation}
where $C<\infty$ is independent of $\kappa$. It follows that $I$ is closed.

Next, observe that the linearization of \eqref{cont_meth} is given by
\begin{multline}
\mathcal{L}(v)=\mathrm{div}\left(\frac{D v}{\sqrt{\eps^2+|D u_{\eps,\sigma,\tau,\kappa}|^2}}-\frac{\langle D u_{\eps,\sigma,\tau,\kappa},D v \rangle D u_{\eps,\sigma,\tau,\kappa}}{\left(\eps^2+|D u_{\eps,\sigma,\tau,\kappa}|^2\right)^{3/2}} \right)\\
-\frac{\kappa\langle D u_{\eps,\sigma,\kappa},D v \rangle}{\left(\eps^2+|D u_{\eps,\sigma,\tau,\kappa}|^2\right)^{3/2}}-\sigma v.
\end{multline}
By the maximum principle and the Hopf lemma the only solution of $\mathcal{L}(v)=0$ with zero boundary conditions is $v=0$. Together with the Fredholm alternative for mixed boundary value problems \cite{Lieb1,Lieb2} and the inverse function theorem, it follows that $I$ is open.

Finally, by standard elliptic estimates, the solution is smooth away from the corners.
\end{proof}

\begin{proof}[{Proof of Theorem \ref{cor tau=0}}]
Existence of solutions $u_{\eps,\sigma}$ for the problem \eqref{tau=0} now follows by taking the solution $u_{\eps,\sigma,\tau}$ of problem \eqref{trip_approx_prob} from Theorem \ref{existence}, and sending $\tau\to 0$. Uniqueness is a consequence of the maximum principle and the Hopf lemma.
\end{proof}

\subsection{Double approximate estimate for $H$}\label{sec: H bound}
The goal of this subsection is to prove a lower bound for $V=H+\sigma u_{\eps,\sigma}$.  We start by giving a uniform sup-bound for $u_{\eps,\sigma}$ in a neighborhood of $\partial K$.
\begin{lemma}\label{supbound close to K0}
There exist constants $\delta_{\star}=\delta_{\star}(K,D)>0$ and $C=C(K,D)<\infty$, such that $u_{\eps,\sigma}\leq C$ in $K\backslash K_{\delta_{\star}}$.
\end{lemma}
\begin{proof}
We will construct a suitable supersolution.
Let $u$ be the arrival time function the free boundary mean curvature flow $\{\partial K_t\}$. Then $u:K\setminus K_{\delta_0}\to \R$, for $\delta_0$ sufficiently small, is smooth and satisfies
\begin{equation}
\textrm{div}\left(\frac{Du}{|Du|}\right)+\frac{1}{|Du|}=0,\qquad \langle N,Du\rangle = 0,
\end{equation}
and
\begin{align}
C^{-1} \leq |Du|\leq C,\quad |D^2u|\leq C,
\end{align}
for some $C<\infty$. For $\delta\in(0,\delta_0)$ consider the function $\phi(t)=\frac{1}{\delta-t}-\frac{1}{\delta}$. A straightforward calculation as in \cite{HH18}*{Lem. 3.8} shows that $\phi(u)$ is a supersolution of (\ref{tau=0}), provided $\delta$ is small enough. Hence, by the maximum principle we conclude that
\begin{equation}
u_{\eps,\sigma}\leq \phi(u)\leq\frac{1}{\delta}\quad \textrm{in}~K_0\backslash K_{\frac{\delta}{2}}.
\end{equation}
This proves the proves the lemma.
\end{proof}

\begin{proposition}\label{V_bd_2}
Let $b=2C_{\partial D}$. Then for $a=a(K,D)<\infty$ sufficiently large, the function $V:M^{\eps,\sigma,\tau}\to \mathbb{R}$ satisfies
\begin{align}
V(x,z) \geq e^{-b\max d} \min \left(\tfrac{1}{2},\min_{\partial K^{\tau}}V\right)e^{-\eps az}.
\end{align}
\end{proposition}
\begin{proof}
Consider the function $Vw$, where $w=\exp(\eps az-bd)$. Suppose towards a contradiction that the minimum of $Vw$ is attained in $\mathring{K}^\tau$ and is less than $\tfrac12 e^{-b\max d}$.
Setting $\kappa=1$ and $m=\eps a$, equations \eqref{dropped_good_term} and \eqref{eq_graphical_id} imply
\begin{multline}\label{ineq_contr_alarge}
0\geq \left(a+\eps^2a^2\right) |\eta^\top|^2+|A|^2+\operatorname{Rc}(\nu,\nu)\\
-b\left(\sigma u  +2a\eps^2V -\operatorname{tr}_{TM}D^2d\right)-\eps^2 a\sigma uV.
\end{multline}
Using $\min V<\frac{1}{2}$ we get $ |\eta^\top|^2\geq \tfrac{3}{4}$ and $
|\sigma u|\leq \tfrac12 + \sqrt{n}|A|$.
Hence, for $a$ sufficiently large, the positive term $\left(a+\eps^2a^2\right) |\eta^\top|^2+|A|^2$ dominates all other terms in \eqref{ineq_contr_alarge}, and  we obtain a contradiction.
\end{proof}

\begin{theorem}\label{grad estim}
There exist constants $a=a(K,D)<\infty$ and $c=c(K,D)>0$ such that for all $\eps,\sigma>0$ we have the estimate
\begin{equation}
H\left(x,\tfrac{1}{\eps} u_{\eps,\sigma}(x)\right)+\sigma u_{\eps,\sigma}(x)\geq ce^{-a u_{\eps,\sigma}(x)}
\end{equation}
for all $x\in K\setminus\partial K$.
\end{theorem}
\begin{proof}
Fix $\eps$ and $\sigma$. In Subsection \ref{Triple approximation}, we have proved that $|Du_{\eps,\sigma,\tau}|\leq C(\eps,\sigma,K)$ in $K^{\tau}$ and that $u_{\eps,\sigma,\tau}$ converges uniformly in $K$ to the unique solution $u_{\eps,\sigma}$ of (\ref{tau=0}) as $\tau\to 0$. Hence, by Lemma \ref{supbound close to K0} we get $u_{\eps,\sigma,\tau}\leq 2\Lambda$ in $K^{\tau}\setminus K_{\delta_{\star}},$ for $\tau$ sufficiently small. Now due to this new $(\eps,\sigma,\tau)$-independent sup-bound, we can choose the constant $\alpha$ in the proof of Lemma \ref{grad bound with tau} (for $\kappa=1$) to be also independent of $\eps,\sigma,\tau$. This in turn implies an $(\eps,\sigma,\tau)$-independent gradient bound for $u_{\eps,\sigma,\tau}$ on $\partial K^{\tau}$. Hence, by Proposition \ref{V_bd_2} we get
\begin{equation}
H\left(x,\tfrac{1}{\eps} u_{\eps,\sigma,\tau}(x)\right)+\sigma u_{\eps,\sigma,\tau}(x)\geq c(K,D)e^{-a u_{\eps,\sigma,\tau}(x)}
\end{equation}
for all $x\in K^{\tau}$. Taking $\tau\to 0$, this estimate passes to the limit in $K\setminus \del K$.
\end{proof}

\subsection{Double approximate estimate for $|A|/H$}
The goal of this section is to prove the following estimate:

\begin{theorem}\label{thm_double_approx_ah}
There exist constants $a=a(K,D)<\infty$ and $C=C(K,D)<\infty$ such that for any $\delta\in (0,\delta_\ast)$ we have
\begin{multline}
\frac{|A|\left(x,\tfrac{1}{\eps}u_{\eps,\sigma}(x)\right)}{H\left(x,\tfrac{1}{\eps}u_{\eps,\sigma}(x)\right)+\sigma u_{\eps,\sigma}(x)}\\
\leq C\left(1+\max_{y\in \partial K_\delta}|A|\left(y,\tfrac{1}{\eps}u_{\eps,\sigma}(y)\right) \right) e^{a u_{\eps,\sigma}(x)}
\end{multline}
for all $x\in K_{\delta}$.
\end{theorem}

To prove Theorem \ref{thm_double_approx_ah} we will apply the maximum principle to the function
\begin{align}\label{eq_def_f}
f:=\frac{|B|+ \Lambda\sigma u+\Theta}{Vw},
\end{align}
where $w=e^{\eps a z-b d}$, and $a,b,\Lambda,\Theta<\infty$ are constants to be chosen below. Here, $B$ denotes a certain perturbation of the second fundamental form, c.f. Edelen \cite{Edelen16}. To define $B$, fix some smooth extensions $k$ and $N$ of the second fundamental form and the unit normal vector of the barrier $\partial D\times \mathbb{R}_+$ to $D\times \mathbb{R}_+$.
The perturbed second fundamental form is then defined by
\begin{align}
B_{ij}=A_{ij}+T_{ij\nu},
\end{align}
where $A$ is second fundamental form of the graph of $u_{\eps,\sigma}/\varepsilon$, and $T$ is a 3-tensor on $D\times \mathbb{R}_+$ defined by
\begin{equation}
T(X,Y,Z)=k(X,Z)\langle Y,N\rangle+k(Y,Z)\langle X,N\rangle.
\end{equation}
Since the graph of $u$ is perpendicular to $\del D\times \R_+$, we have $\langle N,\nu\rangle=0$. Thus, we get $A(X,N)=-k(X,\nu)$ for any tangent vector $X$ perpendicular to $N$, which in turn implies
\begin{equation}\label{equation_b_normal}
B(X,N)=0
\end{equation}
whenever $X$ is perpendicular to $N$.

\begin{lemma}\label{lemma_pert_2ff_normal}
The perturbed second fundamental form satisfies \[N|B|\leq C(|B|+\sigma u+1).\]
\end{lemma}
\begin{proof}
Let $\{e_i\}$ be an orthonormal frame field for $T\del M$ with $e_1=N.$
A calculation as in \cite{Edelen16}*{Lemma 6.1} shows that
\begin{align}
\nabla_Nh_{ij}&=O(|A|+1),\quad \forall i,j>1\\
\nabla_Nh_{11}&=k(\nu,\nu)V+O(|A|+1).
\end{align}
Together with the fact that $b_{1j}=0$ for $j>1$ by \eqref{equation_b_normal}, and the formula  $\nabla T=DT+A\ast T$, c.f. \cite{Edelen16}*{p. 13}, this implies the assertion.
\end{proof}

Moreover, the perturbed second fundamental form controls the second fundamental form. Namely, at any point with $|B|\geq 1$ we have
\begin{equation}\label{bconta1}
|A|\leq C|B|,
\end{equation}
and
\begin{equation}\label{bconta2}
|\nabla A|\leq C(|B|+|\nabla B|).
\end{equation}
Indeed, \eqref{bconta1} follows from the fact that $T$ is bounded, and \eqref{bconta2} follows from the formula $\nabla T=DT+A\ast T$.

\begin{proposition}\label{before applying kato}
Whenever $|B|\geq 1$, then
\begin{multline}\label{key grad barA}
\left(\tfrac{1}{\varepsilon}\nabla_{\eta^\top}-\Delta\right)|B|
 \leq -\tfrac{|\nabla B|^2-|\nabla|B||^2}{|B|}+|B|^3+C(1+\sigma  u)|B|^2.
\end{multline}
\end{proposition}
\begin{proof}
To begin with, by \cite{HH18}*{eq. (4.8), (4.9)} we have
\begin{equation}
\tfrac{1}{\varepsilon}\nabla_{\eta^\top} A=\nabla^2H +\left(A^2-\eps^2\sigma A+\textrm{Rm}(\nu,\cdot,\nu,\cdot)\right) (H+\sigma u).
\end{equation}
Together with Simon's identity
\begin{equation}
\Delta A = \nabla^2 H + A^2H-|A|^2A+O(1+|A|),
\end{equation}
this yields
\begin{equation}\label{evol_a}
\left(\tfrac{1}{\varepsilon}\nabla_{\eta^\top}-\Delta\right)A\leq |B|^2B+C(1+\sigma u)|B|^2.
\end{equation}
Next, the tensor $T$ satisfies the identities
\begin{align}
\tfrac{1}{\varepsilon} \nabla_{\eta^\top} T_{ij\nu}&=\del_kV T_{ijk}-\del_iV T_{\nu j\nu}-\del_jV T_{i\nu\nu}+O(|B|^2),\\
\Delta T_{ij\nu}&=\del_kH T_{ijk}-\del_iHT_{\nu j\nu}-\del_jHT_{i\nu\nu}+O(|B|^2),
\end{align}
c.f. \cite{Edelen16}*{p. 14}. Taking the difference and using $\nabla u =\eps \eta^\top=O(1)$, we infer that
\begin{equation}\label{evol_b}
\left(\tfrac{1}{\varepsilon}\nabla_{\eta^\top}-\Delta\right) T\leq C|B|^2.
\end{equation}
Combining \eqref{evol_a} and \eqref{evol_b} we conclude that
\begin{align}
\tfrac{1}{2}\left(\tfrac{1}{\varepsilon}\nabla_{\eta^\top}-\Delta\right)|B|^2\leq - |\nabla B |^2 +  |B|^4+C(1+\sigma u)|B|^3.
\end{align}
From this the claim follows.
\end{proof}

\begin{lemma}\label{Kato}
There exists a constant $\gamma=\gamma(n)<1$ such that
\begin{align}
|\nabla|B||\leq \gamma|\nabla B|+2|\nabla H|+C(1+|A|).
\end{align}
\end{lemma}
\begin{proof}
The proof is essentially the same as the one of \cite{HH18}*{Prop. 4.12}. The only change is that due to $|\nabla T|=O(1+|A|)$, we need to add $O(1+|A|)$ to the right-hand sides of \cite{HH18}*{eq. (4.16), (4.18)}. The rest of the argument remains intact.
\end{proof}

\begin{proof}[{Proof of Theorem \ref{thm_double_approx_ah}}]
Consider the function $f$ as defined in \eqref{eq_def_f}. The constants $a,b,\Lambda,\Theta$ will be chosen below. Throughout the proof, we tacitly assume that $\eps$ and $\sigma$ are small enough such that
\begin{align}\label{small eps-sigma}
\eps a<1,\quad\sigma \Lambda<1.
\end{align}
Using Lemma \ref{lemma_pert_2ff_normal} we see that
\begin{align}
N(\log f)\leq \frac{C(|B|+\sigma u+1)}{|B|+\Lambda \sigma u + \Theta} - k(\nu,\nu)-b<0,
\end{align}
provided $b$ is large enough. Hence, the maximum of $f$ over $K_\delta$ is either attained in the interior $\mathring{K}_\delta$ or on the Dirichlet boundary $\partial K_{\delta}$. Suppose now the maximum is attained at $x_0\in\mathring{K}_\delta$, and let $z_0=\tfrac{1}{\eps}u_{\eps,\sigma}(x_0)$.

\begin{claim}\label{claim_kato_combined}
There exists $\delta=\delta(n)>0$, such that at $(x_0,z_0)$ we have
\begin{align}\label{evin_b}
\left(\tfrac{1}{\varepsilon}\nabla_{\eta^\top}-\Delta\right)|B|\leq -\delta\tfrac{|\nabla |B||^2}{|B|} + |B|^3+C(1+\sigma u)|B|^2,
\end{align}
provided that $|B|\geq \max \left\{1,\frac{4}{1-\gamma}V\right\}$ at $(x_0,z_0)$.
\end{claim}
\begin{proof}[{Proof of Claim \ref{claim_kato_combined}}]
At $(x_0,z_0)$ we have $\nabla \log f=0$ or equivalently
\begin{align}
\nabla H=\left(\frac{\nabla |B|+\Lambda\sigma\eps\eta^{\top}}{|B|+\Lambda \sigma u+\Theta}-\eps a\eta^{\top}+b\nabla d\right)V-\eps\sigma\eta^{\top}.
\end{align}
Using (\ref{small eps-sigma}) and $|B|\geq \max \left\{1,\frac{4}{1-\gamma}V\right\}$ we deduce that
\begin{align}
|\nabla H|\leq \frac{1-\gamma}{4}|\nabla B|+C| B|.
\end{align}
Therefore, in view of Lemma \ref{Kato} we get the estimate
\begin{align}
|\nabla B|^2-|\nabla|B||^2\geq \delta |\nabla|B||^2-C|B|^2,
\end{align}
where $\delta=\delta(n)>0$. Together with Proposition \ref{before applying kato} this implies the claim.
\end{proof}

Continuing the proof of the theorem, we compute
\begin{align}\label{ev_u}
\left(\tfrac{1}{\varepsilon}\nabla_{\eta^\top}-\Delta\right)u\leq 1,
\end{align}
and
\begin{align}\label{evin_Vw}
\left(\tfrac{1}{\varepsilon}\nabla_{\eta^\top}-\Delta\right)(Vw)
\geq& \left(|B|^2+a-C(|B|+\sigma u)\right)Vw\\
&-2\langle \nabla \log w,\nabla(Vw)\rangle,\nonumber
\end{align}
c.f. Lemma \ref{basic equations}.

Since $0\leq \left(\tfrac{1}{\varepsilon}\nabla_{\eta^\top}-\Delta\right) f -2\langle \nabla \log(Vw),\nabla f\rangle$
at $(x_0,z_0)$, combining the inequalities \eqref{evin_b}, \eqref{ev_u} and \eqref{evin_Vw} we obtain
\begin{multline}
0\leq \frac{1}{Vw}\left(-\delta\tfrac{|\nabla |B||^2}{|B|}+|B|^3+C\sigma u|B|^2\right)+Cf\left(|B|+\sigma u\right)\\
-f\left(|B|^2+a-2\langle \nabla \log w,\nabla\log(Vw)\rangle\right)
\end{multline}
provided that $|B|\geq \left\{ 1,\frac{4}{1-\gamma}V\right\}$. Due to $|\nabla u|\leq 1$, the gradient term can be estimated by
\begin{align}
2f\left| \langle \nabla \log w,\nabla\log(Vw)\rangle\right|=&\frac{2}{Vw}\left|\langle\nabla \log w,\nabla (|B|+\Lambda \sigma u)\rangle\right|\nonumber\\
\leq& \frac{1}{Vw}\left(\delta\tfrac{|\nabla|B||^2}{|B|}+C|B|\right).
\end{align}
Using this, and assuming $|B|\geq \max\{1,\frac{4}{1-\gamma}V,\frac{\sigma}{2n}u\}$, we infer that
\begin{multline}
0\leq (C-\Lambda)\sigma u |B|^2+(C\Lambda-\Theta)|B|^2+(C\Theta-a)|B|,
\end{multline}
which yields a contradiction, provided $\Lambda >C$, $\Theta> C\Lambda$ and $a>C\Theta$.

So far we have shown that there exist constants $a,b,\Lambda,\Theta$, depending only on $D$ and $K$, such that at any interior maximum of $f$ we have
\begin{equation}\label{shownsf}
|B|\leq  \max\left\{1,\frac{4}{1-\gamma}V,\frac{\sigma}{2n}u\right\}.
\end{equation}
The last part of the proof is to show that $f(x_0,z_0)$ is uniformly bounded. To this end, let us point out first that \eqref{shownsf} obviously implies
\begin{equation}\label{shown_so_far}
|B|\leq  \max\left\{1+2|\tr{T}|,\frac{4}{1-\gamma}V,\frac{\sigma}{2n}u\right\}.
\end{equation}

If the maximum on the right hand side of \eqref{shown_so_far} equals $1+2|\tr{T}|$, then we get
\begin{equation}
f(x_0,z_0)\leq \frac{(1+2n\Lambda)(1+2|\tr{T}|) +\Theta}{Vw}\leq C,
\end{equation}
since $Vw$ is bounded below by Theorem \ref{grad estim}. If the maximum on the right hand side of \eqref{shown_so_far} equals $\tfrac{4}{1-\gamma}V$, then
\begin{equation}
f(x_0,z_0)\leq \frac{4+8n}{(1-\gamma)w}+\frac{\Theta}{Vw}\leq C,
\end{equation}
since $w$ and $Vw$ are bounded below. Finally, if the maximum on the right hand side of \eqref{shown_so_far} equals $\tfrac{\sigma}{2n}u$, then $|\operatorname{tr} T|\leq\tfrac{1}{4} \sigma u$ and $|H+\tr T|\leq \frac{\sigma}{2}u$. This implies $V\geq \frac{\sigma}{4}u$, hence again
\begin{equation}
f(x_0,z_0)\leq \frac{4(\Lambda+1)}{w}+\frac{\Theta}{Vw}\leq C.
\end{equation}
Since $\sigma u\leq C$ on $K\setminus K_{\delta_\ast}$ by Lemma \ref{supbound close to K0}, this finishes the proof of the theorem.
\end{proof}

\subsection{Passing to limits and one-sided minimization}

We will first send $\sigma\to 0$, then prove one-sided minimization, and then send $\eps\to 0$.

\begin{lemma}\label{eps limit prop}
There exists a relatively open set $\Omega^\eps \subseteq K$ containing $\partial K$ such that the following holds.
The (improper) limit $u_{\eps,\sigma} \rightarrow u_\eps$ as $\sigma\to 0$ is in $C_{\textrm{loc}}^{0,1}(\Omega^{\eps})\cap C^\infty_{\textrm{loc}}(\Omega^\eps\setminus\partial K)$ and solves
\begin{align}\label{eq_fb_translators}
\operatorname{div}\left(\frac{Du_{\eps}}{\sqrt{\eps^2+\abs{Du_{\eps}}^2}}\right)+\frac{1}{\sqrt{\eps^2+\abs{Du_{\eps}}^2}}&=0  &  \textrm{in} \,\, \mathring{\Omega}^\eps \nonumber \\
\langle N, Du_{\eps}\rangle&= 0 & \textrm{on} \,\, \delta \Omega^\eps\\
u_{\eps}&= 0 & \textrm{on} \,\, \nonumber\partial K \\
u_{\eps}(x)&\to\infty& \,\,x\to \partial \Omega^\eps\setminus \partial K.\nonumber
\end{align}
Moreover, for $x\in K\setminus \partial K$ we have estimates
\begin{align}\label{grad est-eps}
{H}(x,\tfrac{1}{\eps}u_\eps(x))\geq ce^{-au_{\eps}(x)},
\end{align}
where $c=c(K,D)>0$ and $a=a(K,D)<\infty$, and for any $\delta\in(0,\delta_\ast)$ for $x\in K_\delta$ we have the estimate
\begin{align}\label{curv est-eps}
\frac{|A|}{H}(x,\tfrac{1}{\eps}u_\eps(x))\leq C\left(1+\sup_{\partial K_\delta}\frac{|A|}{H}(y,\tfrac{1}{\eps}u_\eps(y))\right)e^{\rho u_\eps(x)},
\end{align}
where $C=C(K,D)<\infty$ and $\rho=\rho(K,D)<\infty$.
\end{lemma}

\begin{proof}
Suppose $\sigma_1\leq \sigma_2$. Since $u_{\eps,\sigma_1}$ is a supersolution of (\ref{tau=0}) for $\sigma=\sigma_2,$ by the maximum principle we have
$u_{\eps,\sigma_1}\geq u_{\eps,\sigma_2}$. Therefore we can take a pointwise (improper) limit $u_\eps(x) = \lim_{\sigma\rightarrow 0} u_{\eps,\sigma}(x)\in [0,\infty]$ for each $x\in K$. Obviously $u_{\eps}=0$ on $\partial K$. Define
\begin{equation}
\Omega^\eps = \{x\in K : u_\eps(x) < \infty\}.
\end{equation}
By Theorem \ref{grad estim}, we have the uniform Lipschitz estimate
\begin{equation}
|De^{-au_{\eps,\sigma}}|\leq C,
\end{equation}
hence $\Omega^\eps\subseteq K$ is open, $u_\eps(x)\to \infty$ as $x\to \partial \Omega^\eps\setminus \partial K$, and \eqref{grad est-eps} holds.
Finally, by standard elliptic estimates (c.f. the proof of Theorem \ref{existence}), the convergence $u_{\eps,\sigma}\to u_\eps$ is in $C^{\infty}_{\textrm{loc}}(\Omega^\eps\setminus\partial K)$, and passing the estimates from Theorem \ref{grad estim} and Theorem \ref{thm_double_approx_ah} to the limit yields \eqref{grad est-eps} and \eqref{curv est-eps}.
\end{proof}

Consider
\begin{equation}\label{def_ell_t}
L^\eps_t := \{(x,z)\in\Omega^\eps \times \mathbb{R}:\eps z\leq u_\eps(x)-t \}.
\end{equation}
The geometric meaning of \eqref{eq_fb_translators} is that $\{L^\eps_t\}$ is a smooth selfsimilar solution of the free boundary mean curvature flow in $D\times \mathbb{R}$, translating downwards with speed $1/\eps$.

\begin{proposition}\label{prop_onesided_min}
The sets $L^\eps_t$ are one-sided minimizing. Namely, given any compact set $E\subset \operatorname{Int}_{\partial D\times \mathbb{R}} L^\eps_0$ we have the estimate
\begin{equation}
\mathcal{H}^{n+1}(\partial L^\eps_t \cap E)\leq \mathcal{H}^{n+1}(\partial F \cap E)
\end{equation}
whenever $ L^\eps_t \subseteq F$ and $F\setminus L^\eps_t\subseteq E$.
\end{proposition}

\begin{proof}
Since $\{L^\eps_t\}_{t>0}$ is mean-convex, we can use the unit normal $\nu$ as calibration and evaluate
\begin{equation}
0\leq\int_{F\setminus L_t^\eps}\operatorname{div}\nu\, d\mathcal{H}^{n+2}
\end{equation}
using the divergence theorem. In general, the boundary of $F\setminus L_t^\eps$ consists of three parts. By the free boundary condition, the part contained in $\partial D\times\mathbb{R}$ does not contribute to the flux integral. From this, the assertion follows.
\end{proof}

\begin{proof}[{Proof of Theorem \ref{thm_ell_reg} (elliptic regularization and consequences)}]
Let $u_\eps:\Omega^{\eps}\to\R$ be as in Lemma \ref{eps limit prop}. By (\ref{grad est-eps}), we have
\begin{align}\label{uniform Lipschitz eps}
|De^{-au_{\eps}}|\leq C.
\end{align}
Thus, for any sequence $\eps_k\to 0$ there is a subsequence $\eps_{k'}\rightarrow 0$ and a relatively open set $\Omega\subseteq K$ containing a neighborhood of $\partial K$ such that $u_{\eps_{k'}}\to \hat{u}$ locally uniformly in $\Omega$ and $u_{\eps_{k'}}(x)\to \infty$ as $x\to \partial \Omega\setminus \partial K$. Since $\hat{u}$ arises as a limit of locally uniform Lipschitz functions, it solves
\begin{align}
\operatorname{div}\left(\frac{D\hat{u}}{\abs{D\hat{u}}}\right)+\frac{1}{\abs{D\hat{u}}}&=0  &  \textrm{in} \,\,\mathring{\Omega}\nonumber \\
\langle N, D\hat{u}\rangle&= 0 & \textrm{on} \,\, \delta \Omega\nonumber\\
\hat{u}&= 0 & \textrm{on} \,\,\partial K
\end{align}
in the viscosity sense. By the definition of viscosity solutions, the family of closed sets $\hat{M}_t:=\{x\in K: \hat{u}(x)=t\}$ satisfies the avoidance principle, and thus is a set-theoretic subsolution of the mean curvature flow with free boundary. Hence, by the same argument as in \cite{HH18}*{p. 1154} the limit $\hat{u}$ agrees with the arrival time function $u$ of the free boundary level set flow. In particular, $\Omega=K\setminus \cap_{t\geq 0}K_t$, and the subsequential convergence entails a full limit.

Recall that $\{L^\eps_t\}$ as defined in \eqref{def_ell_t} is a smooth selfsimilar solution of the free boundary mean curvature flow in $D\times \mathbb{R}$, translating downwards with speed $1/\eps$.
The arrival time function of $\{L^\eps_t\}$ is given by $U_\eps(x,z)=u_{\eps}(x)-\eps z$, and converges locally uniformly to $U(x,z)=u(z)$, which is the arrival time function of $\{K_t\times \mathbb{R}\}$.
Hence, $\{L^\eps_t\}$ converges to $\{K_t\times \mathbb{R}\}$ in the strong Hausdorff sense, c.f. \cite{HK}*{Sec. 4.3}. Moreover, by the compactness theorem for Brakke flow with free boundary \cite{Edelen17}, and by the above uniqueness, $\{\mathcal{H}^{n+1}\lfloor \partial L^\eps_t\}$ converges in the sense of Brakke flows to $\{\mathcal{H}^{n+1}\lfloor \partial K_t\times \mathbb{R}\}$. In particular, $t\mapsto \mathcal{H}^n\lfloor\partial K_t$ is a free boundary Brakke flow. We can also pass the one-sided minimization property from Proposition \ref{prop_onesided_min} to the limit to obtain $\mathcal{H}^n(\partial K_t)\leq \mathcal{H}^n(\partial F)$ whenever $K_t\subseteq F\subseteq K$.
Finally, passing the estimates \eqref{grad est-eps} and \eqref{curv est-eps} to the limit at smooth points via the local regularity theorem \cite{W05,Edelen17} we infer that
\begin{equation}
\inf_{\partial K_t^\textrm{reg}}H\geq c e^{-a t}\qquad (t>0),
\end{equation}
and, taking also into account that $\partial K_\delta$ is smooth with bounded curvature for $\delta>0$ small enough, that
\begin{equation}
\sup_{ \partial K_t^\textrm{reg}}\frac{|A|}{H}\leq C e^{\rho t} \qquad (t\geq \delta).
\end{equation}
This finishes the proof of Theorem \ref{thm_ell_reg}.
\end{proof}

\bigskip

\section{Free boundary Brakke flow and limit flows}\label{sec_free_brakke}
For ease of notation let us pretend that $(N^{n+1},g)=(\mathbb{R}^{n+1},\delta)$ (everything generalizes in a straightforward way to other ambient manifolds).

Let $K_t=F_t(K)$ be a free boundary level set flow with smooth strictly mean-convex initial data $K\subset D$. As before, we write
\begin{equation}
\mathcal{K}=\bigcup_{t\geq 0} K_t\times\{t\}\subset D\times \mathbb{R}_+
\end{equation}
for its space-time track. By Theorem \ref{thm_ell_reg} (elliptic regularization and consequences), we can also consider the associated free boundary Brakke flow
\begin{equation}
\mathcal{M}=\{ \mathcal{H}^n\lfloor\partial K_t \}_{t\geq 0}.
\end{equation}
We recall that the pair $(\mathcal{M},\mathcal{K})$ is called a \emph{mean-convex free boundary flow}.

\subsection{Free boundary Brakke flow}
As in \cite[Sec. 3]{Edelen17}, for any Borel-measurable vector field $X$, define the vector field
\begin{equation}
X_\ast=X-1_{\partial D}\langle X, N\rangle N.
\end{equation}

We recall from \cite[Sec. 4]{Edelen17} that in general a free boundary Brakke flow $\mathcal{M}=\{\mu_t\}_{t\geq 0}$ is given by a family of Radon measures in $D$, that is integer $n$-rectifiable for almost all times, such that
\begin{equation}
\frac{d}{dt}\int \phi \, d\mu_t \leq \int \left( -\phi H_\ast^2 +\langle \nabla \phi, H_\ast\rangle+\partial_t \phi\right) \, d\mu_t
\end{equation}
for all nonnegative $C^1$-functions $\phi$. Here, $\tfrac{d}{dt}$ denotes the limsup of difference quotients, and it is assumed that for almost every time the first variation of the associated varifold $V_{\mu_t}$ is represented by a function $H\in L^2((D,\mu_t);\mathbb{R}^{n+1})$, namely
\begin{equation}\label{firstvar}
\delta V_{\mu_t}(X)=-\int H_\ast \cdot X \, d\mu_t
\end{equation}
for all $C^1$-vectorfields $X$ that are tangential along $\partial D$.\footnote{In particular, by \eqref{firstvar}, $V_{\mu_t}$ has free boundary in the sense of integral varifolds.}

\begin{remark}
Thanks to Theorem \ref{thm_ell_reg} (elliptic regularization and consequences) all free boundary Brakke flows that we encounter in the present paper, or more precisely their stabilized version obtained by crossing with a line, are limits of smooth free boundary flows.
\end{remark}

Let $(\mathcal{M},\mathcal{K})$ be a mean-convex free boundary flow. The \emph{support} of $\mathcal{M}$ consists by definition of all points $X=(x,t)$ which have Gaussian density (see Section \ref{sec_tang_gauss} below) at least one.  We write $\partial \mcfK := \mcfK \setminus \mathrm{Int}_{D \times \R} \mcfK$ for the Dirichlet boundary of $\mcfK$ as a subset of spacetime.

\begin{proposition}[support]\label{prop:support}
If $(\mathcal{M},\mathcal{K})$ is a free boundary Brakke flow then the support of $\mathcal{M}$ is equal to $\partial \mathcal{K}$ for $t > 0$.
\end{proposition}

\begin{proof}
This follows similarly as in \cite[proof of Thm. 5.3]{White_size}.  Alternatively, one use that the stabilized flow $\{M_t \times \R\}_t$ is a limit of smooth, mean-convex flows in $D \times \R$, and argue as in Theorem \ref{thm_limit_flows} below.
\end{proof}

\subsection{Blowup sequences and limit flows}\label{sec_blowup_seq}
Let  $(\mathcal{M},\mathcal{K})$ be a mean-convex free boundary flow in $D$.
Given $X_i=(x_i,t_i)\in\partial \mathcal{K}$ (with $\liminf_{i\to \infty} t_i>0$ and  $\limsup_{i\to \infty} t_i<\infty$), and $\lambda_i\to \infty$, we consider the \emph{blowup sequence} $(\mathcal{M}^i,\mathcal{K}^i)$ which is obtained from $(\mathcal{M},\mathcal{K})$ by translating $X_i$ to the origin and parabolically rescaling by $\lambda_i$. After passing to a subsequence, we can assume that either
\begin{equation}
\lim_{i\to \infty}\lambda_i d(x_i,\partial D) =\infty \qquad \textrm{(interior case)},
\end{equation}
or
\begin{equation}
\lim_{i\to \infty}\lambda_i d(x_i,\partial D) <\infty \qquad \textrm{(boundary case)}.
\end{equation}
Note that $(\mathcal{M}^i,\mathcal{K}^i)$ is defined in the domain $D^i$ which is obtained from $D$ by shifting $x_i$ to the origin and rescaling by $\lambda_i$.
For $i\to\infty$ the domains $D^i$ converge locally smoothly to $\mathbb{R}^{n+1}$ in the interior case, and to a halfspace, which we denote by $\mathbb{H}$, in the boundary case. By the area bounds from one-sided minimization (Theorem \ref{thm_ell_reg}) and the compactness theorem for free boundary Brakke flows \cite[Thm. 4.10]{Edelen17} after passing to a subsequence we can assume that $\mathcal{M}^i$ converges to a limit $\mathcal{M}'$, which is a Brakke flow in the interior case and a free boundary Brakke flow in the boundary case. After passing to a further subsequence we can also assume that $\mathcal{K}^i$ converges in the Hausdorff sense to a limit $\mathcal{K}'$. Any such pair $(\mathcal{M}',\mathcal{K}')$ is called a \emph{limit flow}.

In the boundary case, when $(\mcfM', \mcfK')$ is a free-boundary flow in $\mathbb{H}$, then one can reflect $(\mcfM', \mcfK')$ around the planar barrier to obtain a boundaryless flow $(\widetilde{\mcfM'}, \widetilde{\mcfK'})$ (c.f. \cite[Prop. 4.4]{Edelen17}), called the \emph{reflected limit flow}.

\begin{proposition}[characterization of planar limit flows]\label{prop:planar-limit}
Let $(\mathcal{M},\mathcal{K})$ be a mean-convex free boundary flow in $D$, and let $(\mathcal{M}^i,\mathcal{K}^i)$ be a blowup sequence converging to a limit flow $(\mathcal{M}',\mathcal{K}')$.  Suppose $\spt\mcfM'$ is contained in some static plane. Then one of the following six cases occurs:
\begin{enumerate}
\item $\mcfK'$ is a static half-space in $\R^{n+1}$, and $\mcfM'$ is the static plane $\partial \mcfK'$.
\item $\mcfK'$ is a (quasi)static plane in $\R^{n+1}$, and $\mcfM'$ is the (quasi)static plane $\partial \mcfK'$ with multiplicity two.
\item $\mcfK'$ is a static quarter-space in $\mathbb{H}$, and $\mcfM'$ is the static half-plane $\partial \mcfK'$ with multiplicity one.
\item $\mcfK'$ is a (quasi)static half-plane in $\mathbb{H}$, and $\mcfM'$ is a (quasi)static half-plane $\partial \mcfK'$ with multiplicity two.
\item $\mcfK'$ is a static slab in $\mathbb{H}$ containing $\partial \mathbb{H}$, and $\mcfM'$ is the static plane $\partial \mcfK'$ with multiplicity one.
\item $\mcfK'$ is the (quasi)static plane coincident with $\partial \mathbb{H}$, and $\mcfM'$ is the (quasi)static plane coincident with $\partial \mathbb{H}$ having multiplicity one.
\end{enumerate}
\end{proposition}

\begin{proof}
As in \cite[Thm 5.4]{White_size}, this follows from the fact that $\spt \mcfM^i = \partial \mcfK^i$, the one-sided-minimization of each $\mcfK^i$, and the local regularity theorems of \cite{W05,Edelen17}.
\end{proof}

\begin{figure}[h]
\centering
\includegraphics[scale=0.4]{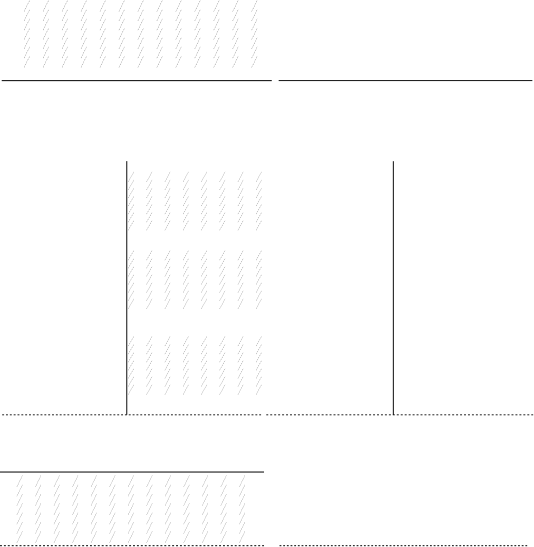}
\caption{Planar limit flows of Proposition \ref{prop:planar-limit}}\label{planar limit flows}
\end{figure}

\begin{theorem}[properties of limit flows]\label{thm_limit_flows}
Let $(\mathcal{M},\mathcal{K})$ be a mean-convex free boundary flow in $D$, and let $(\mathcal{M}^i,\mathcal{K}^i)$ be a blowup sequence converging to a limit flow $(\mathcal{M}',\mathcal{K}')$. Then:
\begin{enumerate}
\item $\mathcal{K}'$ is weakly mean-convex, i.e. $K'_{t_2}\subseteq K'_{t_1}$ whenever $t_2\geq t_1$.
\item The support of $\mathcal{M}'$ equals $\partial \mathcal{K}'$.
\item The sets $\partial \mathcal{K}^i$ Hausdorff converge to $\partial \mathcal{K}'$.
\item $(\mathcal{M}',\mathcal{K}')$ is one-sided minimizing.
\end{enumerate}
\end{theorem}

\begin{proof}
The first assertion is trivial.  Assertions two and three follow as in \cite[proof of Thm. 5.5]{White_size}, using Proposition \ref{prop:planar-limit} in place of \cite[Thm 5.4]{White_size}.  Assertion four follows from the one-sided-minization of $\mcfK^i$ as in \cite[Thm 6.1]{White_size}.
\end{proof}

\subsection{Tangent flows and Gaussian density}\label{sec_tang_gauss}
For $x\in D$ close enough to $\partial D$ there is a well defined projection $\xi(x)$ to the nearest point in $\partial D$. Denote by
\begin{equation}
\tilde{x}=2\xi(x)-x
\end{equation}
the point which is obtained by reflection across $\partial D$. Using this, one can define the almost monotone quantity
\begin{equation}
\Theta_{\textrm{refl}(\partial D)}(\mathcal{M},X,r)
\end{equation}
as in \cite[Def. 5.1.1]{Edelen17}, which interpolates between Huisken's monotone quantity in the interior \cite{Huisken_monotonicity} and an almost monotone reflected quantity close to the boundary.

Let $X$ in the support of $(\mathcal{M},\mathcal{K})$ be fixed, and $\lambda_i\to \infty$. Let $(\mathcal{M}^i,\mathcal{K}^i)$ be the sequence of flows which is obtained from $(\mathcal{M},\mathcal{K})$ by translating $X$ to the origin and parabolically rescaling by $\lambda_i$. Any subsequential limit $(\mathcal{M}',\mathcal{K}')$ is called a \emph{tangent flow at $X$}. By the almost monotonicity formula from \cite[Thm. 5.1]{Edelen17} tangent flows are always backwardly selfsimilar, i.e. $(\mathcal{M}',\mathcal{K}')\cap \{t\leq 0\}$ is invariant under parabolic dilation. In particular, they have a well defined \emph{reflected Gaussian density}
\begin{equation}
\Theta_{\textrm{refl}(\partial D)}(\mathcal{M},X):=\lim_{r\to 0}\Theta_{\textrm{refl}(\partial D)}(\mathcal{M},X,r).
\end{equation}
Tangent flows are either \emph{shrinking}, \emph{static} or \emph{quasistatic}, see \cite{White_stratification}.

If $(\mathcal{M}',\mathcal{K}')$ is a limit flow at $X\in \partial D$, which is defined in a halfspace, and $(\widetilde{\mathcal{M}'},\widetilde{\mathcal{K}'})$ is the doubled flow, then by \cite[Thm. 6.4 and Lem. 7.1]{Edelen17} we have
\begin{equation}
\textrm{Ent}[\widetilde{\mathcal{M}'}]\leq\Theta_{\textrm{refl}(\partial D)}(\mathcal{M},X),
\end{equation}
with equality in the case of tangent flows. Here,
\begin{equation}
\textrm{Ent}[\widetilde{\mathcal{M}'}]=\lim_{t\to - \infty} \int \frac{1}{(4\pi |t|)^{n/2}}e^{-\frac{|x|^2}{4|t|}}d\mu'_{t}(x)
\end{equation}
denotes the \emph{entropy} (aka density at $\infty$) of $\widetilde{\mathcal{M}}'$.

\bigskip

\section{Multiplicity one}\label{sec_mult_one}
The goal of this section is to prove that static and quasistatic density two planes respectively halfplanes cannot occur as tangent flows or limit flows (note also that planes respectively halfplanes of density $\geq 3$ are immediately ruled out by one-sided minimization).

\subsection{Limit flows with entropy at most two}
In this section, we consider the following class of limit flows.

\begin{definition}[class of limit flows]
Let $(\mathcal{M}, \mathcal{K})$ be a mean-convex flow in $D$. Let $\mathcal{C}$ be the class of all limit flows $(\mathcal{M}',\mathcal{K}')$ such that
\begin{enumerate}
\item $\textrm{Ent}[\mathcal{M}']\leq 2$ respectively $\textrm{Ent}[\widetilde{\mathcal{M}}']\leq 2$.
\item $(\mathcal{M}',\mathcal{K}')$ respectively $(\widetilde{\mathcal{M}}',\widetilde{\mathcal{K}}')$ is not a static or quasistatic multiplicity two plane.
\end{enumerate}
\end{definition}

We recall that all limits flows are either free boundary flows in a halfspace or flows without boundary in entire space.
In the above definition, in case $(\mathcal{M}',\mathcal{K}')$ is defined in a halfspace, $(\widetilde{\mathcal{M}}',\widetilde{\mathcal{K}}')$ is reflected flow $\mathbb{R}^{n+1}$.

\begin{proposition}[partial regularity]\label{prop_part_reg}
If $(\mathcal{M}',\mathcal{K}')\in \mathcal{C}$, then no tangent flow at a singular point can be static or quasistatic. In particular, the parabolic Hausdorff dimension of the singular set is at most $n-1$.
\end{proposition}

\begin{proof}
In case $(\mathcal{M}',\mathcal{K}')$ is defined in a halfspace, we consider its double $(\widetilde{\mathcal{M}}',\widetilde{\mathcal{K}}')$. By the equality case of the monotonicity formula and the definition of the class $\mathcal{C}$, no tangent flow can be a static or quasistatic plane of multiplicity two. Hence, by the stratification of the singular set from \cite[Sec. 9]{White_stratification} each time slice of the singular set has Hausdorff dimension at most $n-1$.

Now, if a tangent flow is a stationary cone, then arguing as above we see that its singular set has dimension at most $n-1$. Since, our flow arises as limit of smooth flows with $|A|\leq CH$ by Theorem \ref{thm_ell_reg} (elliptic regularization) using the local regularity theorem, we infer that $A$ vanishes identically on the regular part, i.e. the cone is flat. Furthermore, the cone cannot be union of three or four half-planes by one-sided minimization (Theorem \ref{thm_limit_flows}). Hence, the cone is a static multiplicity one plane, and by the local regularity theorem the point is regular.

Summing up, all tangent flows at singular points are shrinking. Hence, again by \cite[Sec. 9]{White_stratification} the parabolic Hausdorff dimension of the singular set is at most $n-1$.
\end{proof}

\begin{corollary}[static limit flows]\label{cor:staticlimits}
If $(\mathcal{M}',\mathcal{K}')\in \mathcal{C}$ is static (or quasistatic), then one of the following five cases occurs:
\begin{enumerate}
\item $\mcfK'$ is a static half-space in $\mathbb{R}^{n+1}$, and $\mcfM'$ is the static plane $\partial \mcfK'$.
\item $\mcfM'$ is a pair of two static parallel multiplicity one planes in $\mathbb{R}^{n+1}$ and $\mcfK'$ is the region in between.
\item $\mcfK'$ is a static quarter-space in $\mathbb{H}$, and $\mcfM'$ is the static half-plane $\partial \mcfK'$ with multiplicity one.
\item $\mcfM'$ is a pair of static multiplicity one halfplanes in $\mathbb{H}$ with free boundary and $\mcfK'$ is the region in between.
\item $\mcfM'$ is a static multiplicity one plane in $\mathbb{H}$ parallel to the barrier plane $\partial\mathbb{H}$, and $\mcfK'$ is the region in between.
\end{enumerate}
\end{corollary}

\begin{proof}
The argument from above shows that $(\mathcal{M}',\mathcal{K}')$ must be smooth and flat. Hence, it is the union of one or two planes or halfplanes. Together with the one-sided minimization (Theorem \ref{thm_limit_flows}) and unit-regularity, the assertion follows.
\end{proof}
\begin{figure}[h]
\centering
\includegraphics[scale=0.4]{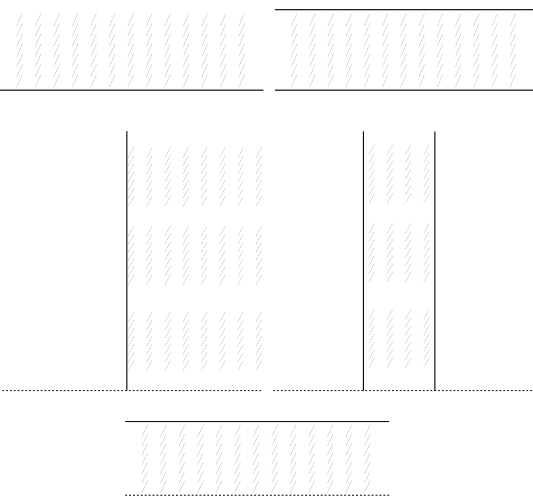}
\caption{(Quasi-)Static limit flows in $\mathcal{C}$}\label{Static limit flows}
\end{figure}

\begin{theorem}[separation theorem]\label{thm:separation}
Let $(\mathcal{M}',\mathcal{K}')\in \mathcal{C}$. In case the flow is defined in a halfspace, suppose there is a halfplane $H$ perpendicular to the barrier plane such that
\begin{equation}
H\subseteq \bigcap_t K_t',
\end{equation}
and suppose the complement of $\cap_t K_t'$ contains points on each side of $H$. Then $(\mathcal{M}',\mathcal{K}')$ is static, and $K'$ is the region between two parallel halfplanes perpendicular to the barrier (a similar statement holds in case the limit flow is defined in entire space and contains a plane).
\end{theorem}

\begin{proof}
Consider the doubled flow $(\widetilde{\mathcal{M}}',\widetilde{\mathcal{K}}')$.
Using the partial regularity result (Proposition \ref{prop_part_reg}) and one sided-minimization (Theorem \ref{thm_limit_flows}), the same argument as in \cite[proof of Thm. 7.4]{White_size} shows that $\widetilde{\mathcal{M}}'$ splits into two components. Since the entropy is at most two, and each nonplanar component contributes strictly more than one, this implies the assertion.
\end{proof}

\subsection{Bernstein-type theorem}
For a set $S\subseteq D$, a point $x\in D$, and a radius $r>0$, the \emph{relative thickness} of $S$ in $B(x,r)$ is
\begin{equation}
\th(S,x,r) =\inf_{|v|=1} \th_v(S,x,r),
\end{equation}
where
\begin{equation}
\th_v(S,x,r) = \frac{1}{r}\sup_{y\in S\cap B(x,r)} |\langle v, y-x\rangle|.
\end{equation}

\begin{lemma}[expanding hole lemma]\label{lemma_exp_hole}
For every $A<\infty$ there exists $\delta=\delta_A>0$ with the following significance.
Suppose $\mathcal{K}$ is a set theoretic subsolution of the free boundary mean curvature flow in our compact domain $D$ or in a halfspace $\mathbb{H}$, or a set theoretic subsolution of the mean curvature flow in $\mathbb{R}^{n+1}$; and suppose that $R>0$ is less than $\diam(D)$ in the former case and arbitrary in the other cases. If there exists $(x,t)$ such that
\begin{equation}
\th(K_t,x,r)<\delta \qquad \textrm{for $r\leq R$}
\end{equation}
and
\begin{equation}
x\notin K_t,
\end{equation}
then
\begin{equation}
\dist(K_{t+r^2},x)\geq A r \qquad \textrm{for $0\leq r \leq \delta R$}.
\end{equation}
\end{lemma}

\begin{proof}
We follow the strategy of the proof of \cite[Thm. 4.1]{White_size}. By translation we may assume $(x,t)=(0,0)$.  Suppose that the result is false for some $A$. Then for every $\delta_j=1/j$ there is a set theoretic subsolution $\mcfK^j$ satisfying the first two conditions but not the last. Let
\begin{equation}
\rho_j = \inf \{r\, |\, \textrm{dist}(K_{r^2}^j,0) \leq Ar\}.
\end{equation}
Since $0\notin K_0^j$, we certainly have $\rho_j>0$, and the failure of the last condition implies that $\rho_j \leq \delta_j R_j$. In particular, in the case where $\mcfK^j$ is defined in $D$, we get $\rho_j\to 0$.

Parabolically rescale by $\rho_j^{-1}$ and pass to a subsequential limit. The limiting domain is either $\mathbb{R}^{n+1}$ or a halfspace $\mathbb{H}$, and the limiting set theoretic subsolution $\mcfK'$ satisfies: $\th(K'_0,0,r)= 0$ for all $r$, as well as $\textrm{dist}(K'_{r^2},0) \geq Ar$ for $r\leq 1$, and $\textrm{dist}(K'_{1},0)=A$.

Since its thickness is zero, $K'_t$ must be contained in a plane $P$, and in the case the domain is $\mathbb{H}$ the halfplane $H=P\cap\mathbb{H}$ must meet $\partial \mathbb{H}$ orthogonally. In either case, the distance condition implies that $K'_{1/2}$ is a proper subset of the static plane or halfplane solution, and must therefore vanish instantly. In particular, $K'_{1} =\emptyset$, which is a contradiction.
\end{proof}

Similarly as in \cite[Sec. 4]{White_size} the expanding hole lemma (Lemma \ref{lemma_exp_hole}) implies the following two corollaries.

\begin{corollary}\label{cor:EH1}
Let $\mcfK$ be as above and assume in addition that it is weakly mean-convex. If $(x,t)$ is a point such that
\begin{equation}
\limsup_{r\to 0} \th(K_t,x,r) < \delta_A,
\end{equation}
and such that $x\notin K_{t+h}$ for all $h>0$, then
\begin{equation}
\liminf_{r\to0} \frac{\dist(K_{t+r^2},x)}{r} \geq A.
\end{equation}
\end{corollary}

\begin{corollary}
\label{cor:EH2}
If $\mathcal{K}$ is a weakly mean-convex set theoretic subsolution of the free boundary mean curvature flow in a halfspace $\mathbb{H}$ or of the mean curvature flow in $\mathbb{R}^{n+1}$, and there is a point $x$ such that
\begin{equation}
\limsup_{r\to\infty} \th(K_0,x,r) < \delta_A,
\end{equation}
then either
\begin{equation}
\liminf_{r\to\infty} \frac{\dist(K_{r^2},x)}{r} \geq A,
\end{equation}
or
\begin{equation}
\bigcap_{t\geq 0} K_t \neq \emptyset.
\end{equation}
\end{corollary}

The following generalizes White's Bernstein-type theorem \cite[Thm. 7.5]{White_size} to limit  flows defined in a halfspace. Importantly, we do not assume a priori at which angle the thin slab meets the barrier.

\begin{theorem}[Bernstein-type theorem]
\label{thm:bern}
There exists an $\eps>0$ with the following significance. If $(\mcfM',\mcfK')\in\mathcal{C}$ is defined in a halfspace $\mathbb{H}$ and there is a point $x$ such that
\begin{equation}
\liminf_{r\to\infty} \frac{\dist(K'_{r^2},x)}{r} <1,
\end{equation}
and
\begin{equation}
\limsup_{r\to\infty} \th(K'_{-r^2},x,r) < \eps,
\end{equation}
then $\mcfM'$ is either: a pair of static parallel multiplicity one half-planes with free boundary in $\mathbb{H}$; or a static multiplicity one plane parallel to the barrier in $\mathbb{H}$. In either case $\mcfK'$ is the region in between the planes of $\mcfM'$ and the barrier.

Similarly, if $(\mcfM',\mcfK')$ is defined in $\mathbb{R}^{n+1}$, then under the same assumptions $\mcfM'$ is a pair of static parallel multiplicity one planes, with $\mcfK'$ the region in between.
\end{theorem}

\begin{proof}
The statement for flows in $\mathbb{R}^{n+1}$ follows from the proof of \cite[Thm. 7.5]{White_size}, so we focus on the case of free boundary flows in $\mathbb{H}$.

Take $\Sigma = \bigcap_t K'_t$, which for small enough $\eps$ must be nonempty by Corollary \ref{cor:EH2}. Consider the flows obtained by translating $(\mcfM',\mcfK')$ by $(0,-T)$ and let $(\mcfM'',\mcfK'')$ be a limit as $T\rightarrow \infty$. Then $(\mcfM'',\mcfK'')$ is a static flow, and $K''=\Sigma$ at any time $t$.

By the classification of static limit flows from Corollary \ref{cor:staticlimits}, we see that $K''$ is either: a static multiplicity two halfplane; the region in between a pair of multiplicity one free boundary halfplanes and the barrier; or finally the region bounded by the barrier plane and a parallel multiplicity one plane. In the first two cases, Theorem \ref{thm:separation} (separation theorem) immediately implies the result. For the final case, the reflected flow $\widetilde{\mcfK''}$ will be the static flow in $\mathbb{R}^{n+1}$ between two parallel multiplicity one planes, so applying the entire case of Theorem \ref{thm:separation} (separation theorem) shows that the reflected flow $(\widetilde{\mcfK'},\widetilde{\mcfM'})$ consists of the region between two multiplicity one planes parallel to the barrier, which implies the result.
\end{proof}

\subsection{Sheeting theorem}
The goal of this section is to prove a sheeting theorem (Theorem \ref{thm_sheeting}, Corollary \ref{cor:sheeting-bd}, Corollary \ref{cor:sheeting-int}).

Let $(\mathcal{M},\mathcal{K})$ be a mean-convex free boundary flow in a compact domain $D$.  Recall that for any $X=(x,t)\in\mathbb{R}^{n+1}\times \mathbb{R}$ and any $r>0$ we denote by
\begin{equation}
B(X,r) = B(x,r)\times (t-r^2,t+r^2)
\end{equation}
the two-sided parabolic ball with center $X$ and radius $r$.\\

The following lemma shows that if a blowup sequence Hausdorff converges to a multiplicity $2$ plane or halfplane, then we can find some sequence of rescaling factors for which one has smooth convergence to a pair of parallel planes or halfplanes.

\begin{lemma}
\label{lem:sheetblow}
Let $(\mcfM^i,\mcfK^i)$ be a blowup sequence (see Section \ref{sec_blowup_seq}), and suppose that
\begin{equation}\label{eq_ass}
d\big(\mcfK^i\cap B(0,2), (P\times \mathbb{R})\cap B(0,2)\big) \rightarrow 0,
\end{equation}
where $P$ is either a plane or a halfplane. Then, there exists some sequence $\rho_i>0$ converging to zero, such that the parabolic dilates $\mathcal{D}_{\rho_i^{-1}} \mcfM^i$ converge smoothly to either:
\item (a) a pair of parallel planes in $\mathbb{R}^{n+1}$ or
\item (b) a pair of parallel halfplanes with free boundary in a halfspace $\mathbb{H}$ or
\item (c) a multiplicity one plane parallel to the boundary of $\mathbb{H}$.

Furthermore, in all cases $\mathcal{D}_{\rho_i^{-1}} \mcfK^i$ converges to the enclosed region.
\end{lemma}

\begin{proof}
Fix $\eps>0$ small enough. Take $\rho_i$ to be the least number such that for all $r\in[\rho_i,1]$ we have
\begin{equation}
\th(K^i_{-r^2},0,r) \leq \eps \quad \textrm{and} \quad \textrm{dist}(K^i_{r^2},0) \leq r.
\end{equation}
By Corollary \ref{cor:EH1} (expanding holes), we have $\rho_i>0$ for each $i$. Moreover,
assumption \eqref{eq_ass} easily implies that $\rho_i\rightarrow 0$.

Consider the dilates $\mathcal{D}_{\rho_i^{-1}}(\mcfM^i,\mcfK^i)$ and take a subsequential limit. Any such limit $({\mcfM}',{\mcfK}')$ satisfies
\begin{equation}
\th({K}'_{-r^2},r)\leq \eps \quad \textrm{and} \quad \textrm{dist}({K}'_{r^2},0) \leq r \quad  \textrm{for all } r\geq 1,
\end{equation}
with at least one inequality being nonstrict for $r=1$, namely
\begin{equation}\label{eq_nonstrict}
\th({K}'_{-1},1)= \eps \quad \textrm{or} \quad \textrm{dist}(K'_1,0) = 1.
\end{equation}
Moreover, by one-sided minimization (Theorem \ref{thm_ell_reg}) the (reflected) density at infinity of $\mcfM'$ is at most $2$. Now due to (\ref{eq_nonstrict}) we can apply Theorem  \ref{thm:bern} (Bernstein-type theorem) to deduce that $\mathcal{M}'$ consists of separate multiplicity $1$ (half-)planes, and then the local regularity theorem \cite{W05,Edelen17} gives smooth convergence of $\mathcal{D}_{\rho_i^{-1}}\mcfM^i$ to ${\mcfM'}$. This proves the lemma.
\end{proof}

The above lemma essentially gives that the $\mcfM^i$ are eventually smooth, but to get smooth convergence at the original scale the strategy is to find a separating surface as follows. Recall that $D^i$ denotes the domain of the rescaled flow $(\mcfK^i, \mcfM^i)$. If we are in case ($a$) or ($b$) of the above lemma, then we let $S^i_t$ be the set of centers of open balls $B$ such that $B\cap D^i \subseteq K^i_t$ and $\bar{B}\cap D^i$ touches $\partial K^i_t$ at two or more points. Set
\begin{equation}
\mathcal{S}^i = \bigcup_t S^i_t \cap D^i.
\end{equation}

\begin{theorem}[sheeting theorem]\label{thm_sheeting}
Let $(\mcfM^i, \mcfK^i)$ be a blowup sequence (see Section \ref{sec_blowup_seq}) and suppose that
\begin{equation}
d\big(\mcfK^i\cap B(0,4), (P\times \mathbb{R})\cap B(0,4)\big) \rightarrow 0,
\end{equation}
where either $\lim D^i = \mathbb{R}^{n+1}$ and $P$ is a plane, or $\lim D^i = \mathbb{H}$ and $P$ is a free boundary halfplane.
Then, for all $i$ large $\mathcal{S}^i \cap B(0,1)$ is a $C^1$-hypersurface that divides $\partial \mathcal{K}^i$ into two nonempty components $\partial \mathcal{K}^i_1$ and $\partial \mathcal{K}^i_2$. In particular, any convergent subsequence of $\mcfM^i$ in $B(0,1)$ converges smoothly to a plane or halfplane with multiplicity two.
\end{theorem}

\begin{proof}
We claim that $\mathcal{S}^i\cap B(0,1)$ is a $C^1$ properly embedded hypersurface of $B(0,1)$, which therefore divides $B(0,1)$ into two disjoint open subsets each bounded by $\mathcal{S}^i$ (once this is shown, $\partial\mathcal{K}^i_1$ and $\partial\mathcal{K}^i_2$ can be defined as the respective portions of $\partial\mcfK^i$ in each open subset).

Suppose the claim is false, so that there are points $X_i \in \mathcal{S}^i\cap B(0,1)$ about which $\mathcal{S}^i$ fails to be a $C^1$ embedded hypersurface. Then we may consider the translates $(\tilde{\mcfM}^i,\tilde{\mcfK}^i):=(\mcfM^i-X_i,\mcfK^i -X_i)$. Up to taking a subsequence, the $\tilde{\mcfK}^i \cap B(0,2)$ will Hausdorff converge locally to a translate of $P\times\mathbb{R}$.

Now, Lemma \ref{lem:sheetblow} (note that outcome ($c$) is excluded by the hypotheses on $P$) implies that for large $i$ there are radii $r_i>0$ for which $\mcfK^i \cap B(X_i,r_i)$ splits as the region in $B(X_i,r_i)$ bounded either by: two smooth, disjoint hypersurfaces $\mathcal{S}^i_1,\mathcal{S}^i_2$ (without boundary); or by smooth disjoint hypersurfaces $\mathcal{S}^i_1,\mathcal{S}^i_2$ (with boundary on $\partial D^i$) together with the barrier $\partial D^i$. In either case, each $\mathcal{S}^i_\alpha$ is graphical over (the plane containing) $P$, with uniformly small $C^{1,1}$ norm, and with $C^0$ norm tending to zero as $i\rightarrow \infty$.

But the distance function from a smooth submanifold $\mathcal{S}$ with boundary is $C^{1,1}$ in a small neighborhood $\mathcal{U}\setminus \mathcal{S}$, indeed with nonzero gradient $\nabla d_\mathcal{S}(x) = \frac{x-\pi_\mathcal{S}(x)}{d_\mathcal{S}}$, see e.g. \cite{MP}. Since $\mathcal{S}^i\cap B(X_i,r_i)$ is clearly given by the locus $d(\cdot, \mathcal{S}^i_1) = d(\cdot, \mathcal{S}^i_2)$ and both $\mathcal{S}^i_\alpha$ Hausdorff converge locally to $P\times \mathbb{R}$, the implicit function theorem then implies that $\mathcal{S}^i \cap B(X_i,r_i)$ is in fact a $C^1$ embedded hypersurface. This provides the desired contradiction, and thus proves the claim.

It remains to show that the $\partial\mathcal{K}^i_\alpha$ converge locally smoothly and separately to $P\times\mathbb{R}$. For each $\alpha=1,2$, by the assumption we have that $\partial\mathcal{K}^i_\alpha \cap B(0,1)$ Hausdorff converges to $(P\times\mathbb{R})\cap B(0,1)$ with some multiplicity. But one-sided minimization (Theorem \ref{thm_ell_reg}) implies that the sum of (reflected) densities is at most $2$. Therefore each must converge with multiplicity $1$, and the partial regularity theorem  \cite{W05,Edelen17} then implies the smooth convergence.
\end{proof}

To reformulate Theorem \ref{thm_sheeting} (sheeting theorem) in a more geometric way, it will be convenient to fix certain boundary-straightening maps $\Phi_y$ centered about a point $y\in\partial D$, as in \cite{Edelen17}. Locally these maps may be described as follows: Up to a translation we may take $y=0$, and up to a rotation and reflection we may take $T_y \partial D$ to be $ \{0\}\times \mathbb{R}^{n}$, with the inwards normal of $\partial D$ at $0$ pointing in positive $x_1$-direction. Near $y$ the barrier $\partial D$ is then locally given by a graph $u$ over $T_y\partial D$, and we define a map $\Phi_y: [0,\rho)\times B^n(\rho)\rightarrow D$ by
\begin{equation}
\Phi_y(s,x') = (u(x'),x') + s\nu_{\partial D} (u(x'),x').
\end{equation}
The key properties are that for every small enough $\rho$, we have uniformly
\begin{equation}
|\Phi- \textrm{id}| \leq C\rho^2,\quad
|D\Phi - \textrm{id}| \leq C\rho ,\quad
|D^k \Phi| \leq C_k\rho^{2-k}.
\end{equation}

In particular, let $D^i = \lambda_i(D-x_i)$ be a sequence of rescalings of $D$, with $\lambda_i d(x_i,\partial D)<\infty$ so that $D^i$ converges to a halfspace $\mathbb{H}$. Let $y_i$ be the nearest point projection of $x_i$ to $\partial D^i$; up to translation we may assume $y=\lim y_i=0$. Then the boundary-straightening maps for $D^i$ are given by $\Phi_i :  [0,\lambda_i\rho)\times B^n(\lambda_i \rho) \rightarrow D^i\hookrightarrow \mathbb{R}^{n+1}$ by
\begin{equation}
\label{eq:blow-up-straighten}
\Phi_i(p) = \lambda_i(\Phi_{y_i}(p/\lambda_i)-x_i).
\end{equation}
Hence,
\begin{equation}
\Phi_i \rightarrow \textrm{id}_{\mathbb{H}} \quad\textrm{locally smoothly.}
\end{equation}

Finally, for a function $f$ defined in $B(0,r)$, define the scale-invariant $C^{2,1}$-norm of $f$ to be the usual $C^{2,1}$-norm of the function
\begin{equation}
(x,t)\in B^{n,1}(0,1)\mapsto rf(x/r, t/r^2).
\end{equation}

\begin{remark}
If for some $r$, $M_t\cap B^{n,1}(y,r)$ is parameterized by
\begin{equation}\label{eq_comp_map}
 B^{n,1}(y,r)\ni (x,t) \mapsto \Phi(f(x,t),x),
\end{equation}
then $f(x,t)$ satisfies the graphical mean curvature flow equation (with respect to $\Phi^*\delta$), given by:
\begin{equation}\label{graphical eq}
 \partial_tf=\gamma^{ij}(Df,f,x)D^2_{ij}f+E,
\end{equation}
where $\gamma$ is the pullback of $\delta$ under the map $x\mapsto \Phi(f(x,t),x)$, and $E$ is an analytic function of $f$, $Df$, $D\Phi$, $D^2 \Phi$ and $D^3 \Phi$, such that $E=0$ when $\Phi=\textrm{id}$.
\end{remark}

Using the above notions, the sheeting theorem (Theorem \ref{thm_sheeting}) implies the following two corollaries.

\begin{corollary}[sheeting at the boundary]
\label{cor:sheeting-bd}
Let $D$ be a compact domain or the halfspace $\mathbb{H}=\{x_1\geq 0 \}$. Assume that $0\in \partial D$, and that the inwards unit normal at $0$ points in positive $x_1$-direction.  Let $\Phi=\Phi_0$ be the boundary-straightening map of $D$ centered at $0$. Then, for any $\eta>0$ there exists an $\eps>0$ with the following significance.

Let $(\mcfM',\mcfK')$ be a free boundary flow in $D$, which is either a mean-convex flow in a compact domain $D$, or a limit flow in $D=\mathbb{H}$. Set
\begin{equation}
\mathcal{H}:= \{(x,t) \in \mathbb{R}^{n+1,1} | x_{n+1}=0, x_1\geq0\},
\end{equation}
and suppose that
\begin{equation}\label{eq:sheeting-supp} d \big( \mcfK'\cap B(0,4r), \mathcal{H}\cap B(0,4r)\big) <\varepsilon r
\end{equation}
for some $r>0$.
Then, there exist functions
\begin{equation}
f,g : B^{n,1}(0,3r) \cap \{x_1\geq 0\} \rightarrow \mathbb{R},
\end{equation}
such that:
\begin{enumerate}
\item $f\leq g$.
\item $f$ and $g$ have scale-invariant $C^{2,1}$ norms $\leq \eta$.
\item Inside $B(0,2r)\cap D$, the set $\mcfK'$ coincides with the region between $\Phi(\graph(f))$ and $\Phi(\graph(g))$.
\item $f,g$ satisfy the graphical mean curvature flow equation (\ref{graphical eq}).
\item $f, g$ satisfy the Neumann boundary condition $\partial_{x_1} f = \partial_{x_1} g=0$ on $B^{n,1}(0,r) \cap \{x_1=0\}$.
\item For any fixed $x$, the functions $t\mapsto f(x,t)$ and $t\mapsto g(x,t)$ are increasing and decreasing respectively.
\end{enumerate}
\end{corollary}

Note that for a mean-convex flow, in order for \eqref{eq:sheeting-supp} to hold for some scale $r$, it must be the case that $r\leq r_0(\eps)$, where $\lim_{\eps\rightarrow 0} r_0(\eps) =0$. With this observation the above corollary follows after scaling from the second case of Theorem \ref{thm_sheeting}. The first case immediately yields:

\begin{corollary}[sheeting in the interior]
\label{cor:sheeting-int}
Let $D$ be a compact domain, or a halfspace $\mathbb{H}$, or entire space $\mathbb{R}^{n+1}$. Assume that $0$ is an interior point of $D$ and let
\begin{equation}
\mathcal{V}:= \{(x,t) \in \mathbb{R}^{n+1,1} | x_{n+1}=0\}.
\end{equation}
Then, for any $\eta>0$ there exists an $\eps>0$ with the following significance. Let $(\mcfM',\mcfK')$ be either a mean-convex free boundary flow in a compact domain $D$ or a limit flow in $\mathbb{H}$ or $\mathbb{R}^{n+1}$. Suppose that
\begin{equation}
d\big( \mcfK'\cap B(0,4r), \mathcal{V}\cap B(0,4r)\big) <\varepsilon r
\end{equation}
for some $r<\frac{1}{4}d(0,\partial D)$.
Then, there exist functions
\begin{equation}
f,g : B^{n,1}(0,3r)  \rightarrow \mathbb{R}
\end{equation}
such that:
\begin{enumerate}
\item $f\leq g$.
\item $f,g$ have scale-invariant $C^{2,1}$ norms $\leq \eta$.
\item Inside $B(0,2r)$, the set $\mcfK'$ coincides with the region between $\graph(f)$ and $\graph(g)$.
\item $f,g$ satisfy the graphical mean curvature flow equation.
\item For any fixed $x$, the functions $t\mapsto f(x,t)$ and $t\mapsto g(x,t)$ are increasing and decreasing respectively.
\end{enumerate}
\end{corollary}

In both cases the strong maximum principle implies that in fact either $f$ is strictly dominated by $g$ or $f$ is identically equal to $g$ (the latter can only happen for limit flows).

Finally, we also have a version parallel to the barrier.
\begin{corollary}[graphs above the barrier]
\label{lem:jacobi-barrier_tang}
Let $(\mcfM,\mcfK)$ be a mean-convex free boundary flow in a compact domain $D$. Assume without loss of generality that $0\in \partial D$, and that the inwards unit normal at $0$ points in positive $x_1$-direction. Denote $\mathbb{H}:=\mathbb{R}^{n+1}\cap \{x_1\geq 0\}$ and
\begin{equation}
\mathcal{V}:= \{(x,t) \in \mathbb{R}^{n+1,1} | x_{1}=0\}.
\end{equation}
Then, for any $\eta>0$ there exists an $\eps>0$ with the following significance. Suppose that
\begin{equation}
d\big( \mcfK\cap B(0,4r), \mathcal{V}\cap B(0,4r)\big) <\varepsilon r
\end{equation}
for some $r>0$.
Then, there exist a function
\begin{equation}
g : B^{n,1}(0,3r)  \rightarrow \mathbb{R}
\end{equation}
such that:
\begin{enumerate}
\item $g\geq 0$.
\item $g$ has scale-invariant $C^{2,1}$ norm $\leq \eta$.
\item Inside $B(0,2r)$, the set $\mcfK$ coincides with the region between $\Phi(graph(0))$ and $\Phi(\graph(g))$, where the graphs are over $\partial{\mathbb{H}}$.
\item $g$ satisfies the graphical mean curvature flow equation (\ref{graphical eq}).
\item For any fixed $x$, the function $t\mapsto g(x,t)$ is decreasing.
\end{enumerate}
\end{corollary}

\begin{proof}
By one-sided minimization (Theorem \ref{thm_ell_reg}), locally $\mathcal{K}$ must be between $\partial D$ and $\partial \mathcal{K}$. From this, the assertion follows easily.
\end{proof}

In particular, we can rule out static density two planes or halfplanes as potential tangent flows.

\begin{corollary}\label{cor_static_tangent_mult}
Static density two planes respectively halfplanes cannot occur as tangent flows.
\end{corollary}
\begin{proof}
Suppose that $X=(x,t)$ is a point of density two with a static (half)plane as tangent flow. Applying Corollary \ref{cor:sheeting-bd}, Corollary \ref{cor:sheeting-int}, or Corollary \ref{lem:jacobi-barrier_tang}, respectively, and using the strong maximum principle (in the first two cases) or the fact that $H>0$  and $\spt\mcfM \subset  D\times \R$ (in the last case),
we see that all points in a neighborhood of $X$ are regular points of multiplicity $1$; a contradiction.
\end{proof}

\subsection{Ruling out density two tangent flows}
In this section, assuming $D$ is mean-convex, we rule out quasistatic density two planes respectively halfplanes as potential tangent flows (recall that the static case has already been ruled out in Corollary \ref{cor_static_tangent_mult}).

For any $X=(x,t)\in\mathbb{R}^{n+1}\times \mathbb{R}$ and any $r>0$ we denote by
\begin{equation}
P(X,r) = B(x,r)\times (t-r^2,t]
\end{equation}
the backwards parabolic ball of radius $r$ with center $X$.

Given any closed spacetime subset $\mcfK'$ of $\mathbb{R}^{n+1}\times \mathbb{R}$, we define two quantities to measure Hausdorff-closeness to a quasistatic plane or halfplane, respectively.
Recall that in \cite[Sec. 9]{White_size}, White defines
\begin{align}\label{def_phi}
&\phi(\mcfK') \textrm{ to be the infimum of $s>0$ such that}\\
&\inf_{\mathcal{V}} d \big(\mcfK' \cap P(0,s^{-1}), \mathcal{V}\cap P(0,s^{-1})\big)<s,\nonumber
\end{align}
where $\inf_\mathcal{V}$ is taken over static planes $\mathcal{V}$ through the origin.
Similarly, we define
\begin{align}\label{def_phi+}
&\phi_+(\mcfK')  \textrm{ to be the infimum of $s>0$ such that}\\
 &\inf_{\mathcal{H}} d(\mcfK' \cap P(0,s^{-1}), \mathcal{H}\cap P(0,s^{-1}))<s,\nonumber
 \end{align}
where $\inf_\mathcal{H}$ is taken over static halfplanes $\mathcal{H}=\mathcal{V}\cap \{ x_1 \geq 0\}$ that intersect $\{x_1=0\}$ orthogonally at 0 (that is, $(0,0)\in \partial\mathcal{H}$ and the inner conormal is $e_1$). \\

The following lemma gives sheeting sequences and their limiting behaviour, for blowups at the boundary:

\begin{lemma}[sheeting sequence at the boundary]
\label{lem:jacobi}
Let $D$ be a compact domain. Assume without loss of generality that $0\in \partial D$, and that the inwards unit normal at $0$ is given by $e_1$. Let $(\mcfM^i,\mcfK^i)$ be either a blowup sequence at $0\in \partial D$ or a sequence of tangent flows at $0$, and suppose that $\phi_+(\mcfK^i)\rightarrow 0$.

Then, for large enough $i$ there are functions $f_i$ and $g_i$, defined on an exhaustion of $\mathbb{R}^n_+ \times (-\infty,0)$, such that:\footnote{Here, we denote $\mathbb{R}^n_+=\{x\in \mathbb{R}^n\, | \, x_1\geq 0\}$ and $\mathbb{H}=\{x\in \mathbb{R}^{n+1}\, | \, x_1\geq 0\}$.}
\begin{enumerate}
\item Either $f_i<g_i$ everywhere, or $f_i \equiv g_i$.
\item For any $U\subset\subset \mathbb{H}$ and $[a,b]\subset (-\infty,0)$, for $i$ large enough the region $K^i_t$ coincides in $U$ with the region between $\Phi_i(\graph (f_i))$ and $\Phi_i(\graph (g_i))$ for all $t\in [a,b]$, where $\Phi_i$ denotes the boundary straightening map for $D^i = \lambda_i D$ as in (\ref{eq:blow-up-straighten}).
\item $f_i$ and $g_i$ converge smoothly on compact subsets to $0$.
\item $f_i$ and $g_i$ solve the graphical mean curvature flow equation in the pullback metric $\Phi_i^* \delta$.
\item $f_i$ and $g_i$ satisfy the zero Neumann boundary condition.
\item $f_i$ and $g_i$ are increasing and decreasing in time, respectively.
\end{enumerate}

Furthermore, if $f_i<g_i$ for infinitely many $i$, then there exist constants $c_i>0$ and a subsequence $c_i\cdot(g_i-f_i)$ that converges smoothly on compact subsets to the constant function $u(x,t)\equiv 1$ on $\mathbb{R}^n_+ \times(-\infty,0)$.
\end{lemma}

A similar statement holds for blowups in the interior:

\begin{lemma}[sheeting sequence in the interior]
\label{lem:jacobi-int}
Let $D$ be a compact domain. Let $(\mcfM^i,\mcfK^i)$ be either a blowup sequence about a fixed interior point or a sequence of tangent flows at a fixed interior point, and suppose that ${\phi}(\mcfK^i)\rightarrow 0$.

Then, for large enough $i$ there are functions $f_i$ and $g_i$, defined on an exhaustion of $\mathbb{R}^n \times (-\infty,0)$, such that:
\begin{enumerate}
\item Either $f_i<g_i$ everywhere, or $f_i \equiv g_i$.
\item For any $U\subset\subset \mathbb{R}^{n+1}$ and $[a,b]\subset (-\infty,0)$, for $i$ large enough the region $K^i_t$ coincides in $U$, after a suitable rotation, with the region between $\graph (f_i)$ and $\graph (g_i)$ for all $t\in [a,b]$.
\item Both sequences $f_i$ and $g_i$ converge smoothly on compact subsets to 0.
\item $f_i$ and $g_i$ are solutions to the graphical mean curvature flow equation.
\item $f_i$ and $g_i$ are increasing and decreasing in time, respectively.
\end{enumerate}

Furthermore, if $f_i<g_i$ for infinitely many $i$, then there exist constants $c_i>0$ and a subsequence $c_i(g_i-f_i)$ that converges smoothly on compact subsets to the constant function $u(x,t)\equiv 1$ on $\mathbb{R}^n \times(-\infty,0)$.
\end{lemma}

\begin{proof}[Proof of Lemma \ref{lem:jacobi} and Lemma \ref{lem:jacobi-int}]
Taking $i \to \infty$, we get convergence to some tangent flow $(\mcfM', \mcfK')$ which is either: defined in $\mathbb{H}$ and supported in a free-boundary half plane; or defined in $\mathbb{R}^{n+1}$ and supported in a plane. By Proposition \ref{prop:planar-limit} these must be the (quasi-)static flows with multiplicity 2. We can therefore apply Corollary \ref{cor:sheeting-int} respectively Corollary \ref{cor:sheeting-bd} to obtain the required $f_i, g_i$.

Suppose now there is a subsequence with $f_i<g_i$. Since $\Phi_i \rightarrow \textrm{id}$ locally smoothly, the difference
\begin{equation}
u_i := g_i-f_i>0
\end{equation}
satisfies a linear parabolic equation with coefficients converging locally smoothly as $i\rightarrow \infty$ to those of the ordinary heat equation.

Let us first analyze the interior case. Since the functions $u_i$ on the one hand are decreasing by mean-convexity, but on the other hand want to become increasing driven by the Harnack inequality, the argument can be concluded as in the proof of \cite[Thm. 9.1]{White_size}.

In the boundary case, we consider the sequence of functions $\tilde{u}_i$ which is obtained from $u_i$ via doubling at the boundary of $\mathbb{R}^n_+$. Applying the same argument to $\tilde{u}_i$ the proof can be concluded in this case also.
\end{proof}

Finally, we also have a version for blowups parallel to the barrier.  Before stating it we need an auxilary result concerning mean-convex domains.

\begin{lemma}\label{lem:min-surface-outside}
Assume $D$ is mean-convex. Then for any point $x \in \del D$, there exists $r > 0$ and a smooth minimal surface in $B_r(x) \setminus \operatorname{int} D$ which passes through $x$.
\end{lemma}

\begin{proof}
Assume without loss of generality that $x=0$, $\mathbb{R}^n = T_0\partial D$. The result follows from the implicit function theorem: Consider for instance the map from $C^{2,\alpha}(B_1) \times \mathcal{S}^2_+(B_1)$ that maps \[(u,g)\to (u(0),Du(0),D^2u(0), H_g u),\] where $H_g u \in C^\alpha(B_1)$ is the mean curvature of $\graph (u)$ in the metric $g$. The linearization in $u$ at the Euclidean disk $(0,\delta)$ is clearly surjective, so in particular for any diagonal, trace-free matrix $A$ with $|A|<\varepsilon$ (and any metric $|g-\delta|<\varepsilon$) we are able to find a smooth ($g$-)minimal surface $\Sigma$ tangent to $D$ with second fundamental form equal to $A$ at $x$.

To complete the proof, let $B$ be the second fundamental form of $\partial D$ at $x$, then $B$ is diagonal in some orthonormal basis and has $\tr(B)>0$, so we may fix a diagonal, trace-free matrix $A<B$ and apply the above after appropriate scaling. Note that $A<B$ ensures that $\Sigma$ is disjoint from $\operatorname{int} D$ in a small enough ball.
\end{proof}

\begin{lemma}[graphical sequence above the barrier]
\label{lem:jacobi-barrier}
Let $D$ be a compact domain. Assume without loss of generality that $0\in \partial D$. Let $(\mcfM^i,\mcfK^i)$ be a blowup sequence at $0$, or a sequence of tangent flows at $0$, and suppose that ${\phi}(\mcfK^i)\rightarrow 0$.  Let $\Sigma_i$ be the corresponding dilates of the minimal surface passing through $0$ as constructed in Lemma \ref{lem:min-surface-outside}, or if the $(\mcfM^i, \mcfK^i)$ are tangent flows take $\Sigma_i \equiv \partial \mathbb{H}$.

Then, for large enough $i$ there are functions $g_i$, defined on an exhaustion of $\partial \mathbb{H} \times (-\infty,0)$, and $f_i$ defined on an exhaustion of $\partial \mathbb{H}$, where $\mathbb{H}=\{x\in \mathbb{R}^{n+1} \, |\, x_1\geq 0\}$, such that setting $f_i(\cdot,t)\equiv f_i(\cdot)$ we have:
\begin{enumerate}
\item Either $g_i > 0 \geq f_i$, or $g_i \equiv 0 \equiv f_i$.
\item For any $U\subset\subset \mathbb{H}$ and $[a,b]\subset (-\infty,0)$, for $i$ large enough the region $K^i_t$ coincides in $U$ with the region between $\Phi_i(\graph (0))$ and $\Phi_i(\graph (g_i))$, where the graphs here are over $\partial \mathbb{H}$.
\item For any $U \subset\subset \mathbb{R}^{n+1}$, for $i$ large enough $\Sigma_i$ coincides in $U$ with $\Phi_i(\graph(f_i))$.
\item $g_i$ and $f_i$ converge smoothly on compact subsets to 0.
\item $g_i$ and $f_i$ are solutions to the graphical mean curvature flow equation in the pullback metric $\Phi_i^* \delta$.
\item $g_i$ is decreasing in time.
\end{enumerate}
Furthermore, if $g_i > 0$ for infinitely many $i$, then there exist constants $c_i>0$ and a subsequence $c_i (g_i - f_i)$ that converges smoothly on compact subsets to the constant function $u(x,t)\equiv 1$ on $\partial\mathbb{H}\times(-\infty,0)$.
\end{lemma}

\begin{proof}
Any subsequential limit of the $\mcfK^i$ must be a plane, and hence must be the barrier plane $\partial \mathbb{H}$ with multiplicity 1 (reflected density 2). The existence of the required $g_i$ follows from Corollary \ref{lem:jacobi-barrier_tang}, and the $f_i$ exist since $\Sigma_i$ is smooth, and converging smoothly to the barrier plane at $0$.  The convergence of $c_i(g_i - f_i)$ follows like in the interior and free boundary cases.
\end{proof}

The following theorem shows that density two planes respectively halfplanes are isolated:

\begin{theorem}[isolation]
\label{thm:effective-sep}
Let $(\mcfM,\mcfK)$ be a mean-convex flow in a compact domain $D$, or a limit flow of such a flow. In case $0\in \partial D$, assume without loss of generality that the inwards unit normal at $0$ points in positive $x_1$-direction. Then, there exists $\delta>0$ such that for any tangent flow $(\mcfM',\mcfK')$ to $(\mcfM,\mcfK)$ at $X=(0,t)$, we have:
\begin{enumerate}
\item If $0$ is a boundary point, and $\phi_+(\mcfK')<\delta$, then $\phi_+(\mcfK')=0$.
\item If $0$ is an interior point, and $\phi(\mcfK')<\delta$, then $\phi(\mcfK')=0$.
\item If $0$ is a boundary point, and $\phi(\mcfK')<\delta$, then $\phi(\mcfK')=0$.
\end{enumerate}
\end{theorem}

\begin{proof}
Suppose towards a contradiction that there is a sequence of tangent flows $(\mcfM^i, \mcfK^i)$ at $X=(0,t)$ with ${\phi}(\mcfK^i)\rightarrow 0$ or ${\phi}_+(\mcfK^i)\rightarrow 0$, respectively. We must show that for large enough $i$, we have $\phi(\mcfK^i)=0$, respectively ${\phi}_+(\mcfK^i)= 0$.

In each case (1)--(3) consider the functions $f_i$ and $g_i$ given by Lemma \ref{lem:jacobi}, Lemma \ref{lem:jacobi-int} and Lemma \ref{lem:jacobi-barrier}, respectively, taking $f_i=0$ for case (3). Since we are dealing with sequences of tangent flows, all boundary-straightening maps are trivial, so $\mcfK^i$ corresponds to the region between $\graph(f_i)$ and $\graph(g_i)$. Moreover, since the $(\mcfK^i,\mcfM^i)$ are tangent flows, they are backwardly selfsimilar, i.e. it holds that
\begin{equation}\label{selfsimi_tang_seq}
f_i(rx,r^2 t) = rf_i(x,t)\,\, \textrm{ and } \, \, g_i(rx,r^2 t) = rg_i(x,t)
\end{equation}
for all $r>0$ and all $(x,t)$ with $t<0$.

If $f_i<g_i$ for infinitely many $i$, then by the conclusions of the lemmata above, there exist $c_i>0$ so that $c_i(g_i-f_i)$ converges smoothly to the constant function $u(x,t)\equiv 1$. However, equation \eqref{selfsimi_tang_seq} implies that $u(rx,r^2t) = ru(x,t)$, which is absurd.

So $f_i\equiv g_i$ for all sufficiently large $i$. But then $f_i=g_i$ are both increasing and decreasing, and hence constant in $t$. The selfsimilarity above then states that
\begin{equation}
g_i(rx) = rg_i(x).
\end{equation}
Since $g_i$ is smooth and 1-homogenous, it must be linear.
In case (2) and (3), we conclude that $\mcfK^i$ is a plane for $i$ large enough, hence $\phi(\mcfK^i)=0$. In case (1),  taking also into account the vanishing Neumann boundary data, we conclude that $\mcfK^i$ is a halfplane orthogonal to $\partial \mathbb{H}$ for $i$ large enough, hence $\phi_+(\mcfK^i)=0$. This finishes the proof of the theorem.
\end{proof}

We can now rule out quasistatic density two tangent flows:

\begin{theorem}[tangent flows]
\label{thm:quasi-tangent-flows}Quasistatic multiplicity two planes, respectively halfplanes, cannot occur as tangent flows to a mean-convex free boundary flow $(\mcfM,\mcfK)$. If additionally $D$ is mean-convex, then quasistatic density two planes also cannot occur as tangent flows.
\end{theorem}

Note the last statement excludes case (6) of Proposition \ref{prop:planar-limit}.

\begin{proof}
We follow the proof strategy in \cite[Thm. 9.2]{White_size}, with some adjustments.

First suppose, for the sake of contradiction, that one tangent flow at a point $X = 0$ is a quasistatic multiplicity two plane, respectively halfplane.
In particular, we are in case (4) or (2) respectively of Proposition \ref{prop:planar-limit}.
Recall that by one-sided minimization, any planar limit flow must have (reflected) density either 1 or 2. In particular, if a limit flow is a multiplicity 2 plane, then it must be defined in $\mathbb{R}^{n+1}$ (that is, it cannot be a plane parallel to the barrier in $\mathbb{H}$).

We may assume without loss of generality that $X=0$, and in case $0$ is a boundary point that the inwards normal of $\partial D$ points in positive $x_1$-direction.
Consider the parabolic dilations $\mathcal{D}_\lambda \mcfK$. By the discussion above, we are in case (1) or (2) of Theorem \ref{thm:effective-sep}, so we must have
\begin{equation}
\lim_{\lambda\to\infty}{\phi}(\mathcal{D}_\lambda \mcfK) =0\quad  \textrm{ or }  \quad \lim_{\lambda\to\infty}{\phi}_+(\mathcal{D}_\lambda \mcfK) =0,
\end{equation}
respectively, at $X=0$. In particular, by definition of $\phi$ and $\phi_+$, every tangent flow at $X=0$ must be of the same type; that is, a quasistatic multiplicity two plane, respectively halfplane.

In either case we may consider the quantity
\begin{equation}
V(r) = \sup\{\rho\, | \, 0\in (\overline{B^{n+1}(y,\rho)}\cap D') \subset K_{-r^2} \text{ for some } y\},
\end{equation}
where $D'$ is $\mathbb{R}^{n+1}$ in the planar case and $\mathbb{R}^{n+1}\cap \{x_1\geq 0\}$ in the halfplanar case. The observation above implies that
\begin{equation}
\lim_{r\to 0} \frac{V(r)}{r} =0.
\end{equation}
In particular, we may take a  sequence $r_i\to 0$ such that
\begin{equation}
\frac{V(r_i)}{r_i} < 2\frac{V(3r_i)}{3r_i}.
\end{equation}
Consider the parabolic dilates $\mcfK^i = \mathcal{D}_{r_i^{-1}}\mcfK$. Applying Lemma \ref{lem:jacobi} and Lemma \ref{lem:jacobi-int}, respectively, gives functions $f_i \leq g_i$ such that $\mcfK^i$ corresponds to the region between $\Phi_i(\graph (f_i))$ and $\Phi_i(\graph (g_i))$, for some diffeomorphisms $\Phi_i$ converging smoothly to $\textrm{id}$. In particular, for large $i$ we have $V(\lambda r_i) \simeq u_i(0,-\lambda^2)$, where $u_i:=g_i-f_i$, and hence
\begin{equation}\label{incom}
u_i(0,-1) \leq \frac{3}{4} u_i(0,-9).
\end{equation}
Since $(\mcfM,\mcfK)$ is a mean-convex flow, we must have $u_i>0.$ Therefore by the above lemmata there exist $c_i>0$ so that $c_iu_i$ converges uniformly to $u(x,t)\equiv 1$. This is incompatible with the inequality (\ref{incom}).

For the remaining density $2$ case, suppose a tangent flow at $X = 0$ is a quasistatic multiplicity one plane coincident with the barrier tangent plane $\partial\mathbb{H}$.  Let $\Sigma$ be the minimal surface in $B_r(0) \setminus \operatorname{int} D$ passing through $0$ as in Lemma \ref{lem:min-surface-outside}, for some $r > 0$.  Set $\mcfM_\Sigma$ to be the mean curvature flow in $B_r(0)$ obtained by keeping $\Sigma$ static, and let $\mcfK_\Sigma$ be the region in $B_r(0)$ between $\Sigma$ and $D$.

By Theorem \ref{thm:effective-sep} (isolation), every tangent flow of $\mcfM$ at $X = 0$ coincides with $\partial\mathbb{H}$, with multiplicity-one. Therefore by Corollary \ref{lem:jacobi-barrier_tang}, there is a parabolic region
\[
\mathcal{U} = \{ |x|^2 \leq |t| , \quad t_0 < t < 0 \} \quad \text{ for some $t_0 < 0$},
\]
so that in $\mathcal{U}$ the flow $(\mcfM + \mcfM_\Sigma, \mcfK \cup \mcfK_\Sigma)$ is a graphical, weakly mean-convex mean curvature flow without boundary, with the property that any blow-up sequence centered at $(0, 0)$ converges smoothly to a multiplicity two plane.  We can then use the same argument as in the interior, using Lemma \ref{lem:jacobi-barrier} in place of Lemma \ref{lem:jacobi-int}, case to deduce a contradiction.
\end{proof}

\subsection{Ruling out density two limit flows}
In this section, for mean-convex free boundary flow $(\mathcal{M},\mathcal{K})$ in a mean-convex domain $D$, we rule density 2 (quasi)static planes or halfplanes as potential limit flows (in the case of tangent flows this has already been done in the previous section). Adapting \cite[Sec. 12]{White_size}, we start with the following lemma:

\begin{lemma}
\label{thm:limit-minimal}
Let $(\mcfM',\mcfK')$ be a limit flow of $(\mathcal{M},\mathcal{K})$, defined in the limiting domain $D'=\mathbb{R}^{n+1}$ or $D'=\mathbb{H}$. Let $(\mcfM'',\mcfK'')$ be a tangent flow of $(\mcfM',\mcfK')$ taken at a density two point $X=(x,t)$. Then:
\begin{enumerate}
\item If $x$ is an interior point of $D'$, and $\mcfK''$ is a static or quasistatic plane, then there is an open neighborhood $U$ of $x$ in $\mathbb{R}^{n+1}$, an open interval $(a,b)$, and a properly embedded smooth minimal hypersurface $\Sigma$ in $U$, such that $K'_\tau \cap U = \Sigma$ for all $\tau\in(a,b)$.
\item If $x$ is a boundary point of $D'$, and $\mcfK''$ is a static or quasistatic halfplane, then there is an open neighborhood $U$ of $x$ in $\mathbb{H}$, an open interval $(a,b)$, and a properly embedded smooth free boundary minimal hypersurface $\Sigma$ in $U$, such that $K'_\tau \cap U = \Sigma$ for all $\tau\in(a,b)$.
\item If $x$ is a boundary point of $D'$, and $\mcfK''$ is a static or quasistatic plane, then there exists a neighborhood $U$ of 0 in $\mathbb{H}$, and an open interval $(a,b)$, such that $K'_\tau \cap U = \partial \mathbb{H}\cap U$ for all $\tau\in (a,b)$.
\end{enumerate}
Furthermore, in case (1) we have $\Theta(\mathcal{M}',\cdot)\geq 2$ on all of $\Sigma\times (-\infty,b]$, and in cases (2), (3) we have $\Theta(\widetilde{\mathcal{M}'},\cdot)\geq 2$ on all of $\widetilde{\Sigma}\times (-\infty,b]$ or $\partial \mathbb{H}\times (-\infty,b]$ respectively.
\end{lemma}

\begin{proof}
We proceed as in the proof of Theorem \ref{thm:quasi-tangent-flows} (tangent flows). In the limit flow setting, we have $D'=\mathbb{R}^{n+1}$ or $D'=\mathbb{H}$, so the dilates $\mathcal{D}_{r_i^{-1}}\mcfK'$ correspond in case (1) and (2) to the region between the graphs of $f_i \leq g_i$. However, it is possible to have $f_i\equiv g_i$ for limit flows, and in fact the proof of Theorem \ref{thm:quasi-tangent-flows} shows that this must be the case for sufficiently large $i$. This immediately implies that, in some backwards parabolic neighborhood of $X$, the flow $\mcfK'$ is a smooth static (free boundary) mean curvature flow, which yields (1) and (2).

Similarly, in case (3) we see that $g_i\equiv 0$ for large enough $i$, which yields that $\mcfK'$ equals $\partial \mathbb{H}$ in a  backwards parabolic neighborhood of $X$.

The final assertion follows from arguing similarly as in \cite[proof of Thm. 12.2]{White_size}.
\end{proof}

\begin{theorem}[limit flows]
Let $(\mcfM,\mcfK)$ be a mean-convex free boundary flow in a mean-convex domain $D$. Then (quasi)static density two planes or halfplanes do not occur as limit flows.
\end{theorem}

\begin{proof}

We follow the proof strategy of \cite[Thm. 12.3]{White_size}, with some adjustments.   Suppose towards a contradiction that there is a blowup sequence $(\mcfM^i,\mcfK^i)$ that converges to a (quasi)static density two plane or halfplane $(\mcfM^\infty,\mcfK^\infty)$.
Fix $\delta>0$ as in Theorem \ref{thm:effective-sep} (isolation).

Adjusting the sequence a bit we can assume that $0\in \partial D^i \cap \partial K_0^i$, and that the inwards unit normal of $\partial D^i$ at $0$ is given by $e_1$.

\bigskip

Case 1: $(\mcfM^\infty,\mcfK^\infty)$ is a density two halfplane.\\

Let $\mu_i>0$ be the smallest number for which
\begin{equation}\label{eq_smallest_bla}
\phi_+(\mathcal{D}_{\mu_i} \mcfK^i) \geq \delta/2.
\end{equation}
Note that $\mu_i\to 0$ by hypothesis. Let  $(\mcfK',\mcfM')$ be a subsequential limit of $(\mathcal{D}_{\mu_i} \mcfK^i,\mathcal{D}_{\mu_i} \mcfM^i)$.
The limit satisfies
\begin{equation}
\phi_+(\mcfK') \geq \delta/2,
\end{equation}
but
\begin{equation}\label{eq_how_we_rescaled}
\phi_+(\mathcal{D}_\mu \mcfK') \leq \delta/2\quad \textrm{ for } \mu>1.
\end{equation}
By equation \eqref{eq_how_we_rescaled} and Theorem \ref{thm:effective-sep} (isolation) any tangent flow $(\mcfM'',\mcfK'')$ to $(\mcfM',\mcfK')$ at the origin must be a multiplicity 2 halfplane. In particular, by Theorem \ref{thm:quasi-tangent-flows}, $(\mcfM',\mcfK')$ must be a limit flow, not a dilate of $(\mcfM,\mcfK)$.

Now, by Lemma \ref{thm:limit-minimal} there exists a $b\in \mathbb{R}$ and a free boundary minimal hypersurface $\Sigma$ in an open neighborhood $U\subset\mathbb{H}$ of the origin, so that
\begin{equation}
\Theta(\widetilde{\mcfM'},\cdot)\geq 2\quad \textrm{on } \widetilde{\Sigma} \times (-\infty,b].
\end{equation}
Consider the time-translates of $(\widetilde{\mcfM'},\widetilde{\mcfK'})$ by $(x,t)\to (x,t+j)$; sending $j\to \infty$ we get a static limit flow $(\widehat{\mcfM},\widehat{\mcfK})$, with
\begin{equation}
\widehat{K}= \widehat{K}_t \equiv \bigcup_\tau K'_\tau,
\end{equation}
and
\begin{equation}
\Theta(\widehat{\mcfM},\cdot) \geq 2\quad \textrm{on } \Sigma\times \mathbb{R}.
\end{equation}

Together with \eqref{eq_smallest_bla} and monotonicity this implies
\begin{equation}
\textrm{Ent}[\widehat{\mcfM}]>2.
\end{equation}

\bigskip

Case 2: $(\mcfM^i,\mcfK^i)$ converges to a (quasi)static density 2 plane.\\

When scaling down along the sequence there is the potential scenario that the barrier comes back in from infinity. To deal with this, instead of $\phi_+$, we consider the more general quantity
\begin{align}\label{def_tilde_phi}
&\phi_{h}(\mcfK')  \textrm{ to be the infimum of $s>0$ such that }\\
&\inf_{\mathcal{H}} d(\mcfK' \cap P(0,s^{-1}), \mathcal{H}\cap P(0,s^{-1}))<s,\nonumber
 \end{align}
where $ \inf_{\mathcal{H}}$ is now taken over all halfplanes $\mathcal{H}=\mathcal{V}\cap \{x_1 \geq a\}$, $a\leq 0$, and $\mathcal{V}$ is a static plane intersecting $\{x_1=0\}$ orthogonally at 0 (in particular, $\mathcal{H}\ni (0,0)$ and the inner conormal on $\partial \mathcal{H}$ is given by $e_1$). \\

Consider the quantity
\begin{equation}
\psi = \min(\phi, \phi_{h}),
\end{equation}
where $\phi$ is defined in \eqref{def_phi} and $\phi_{h}$ is defined in \eqref{def_tilde_phi}.
Let $\mu_i>0$ be the smallest number for which
\begin{equation}
\psi(\mathcal{D}_{\mu_i} \mcfK^i) \geq \delta/2.
\end{equation}
Note that $\mu_i \rightarrow 0$.  Let $(\mcfK',\mcfM')$ be a subsequential limit, defined in $D'$, of $\mathcal{D}_{\mu_i}(\mcfK^i,\mcfM^i)$.
The limit satisfies
\begin{equation}
\psi(\mcfK') \geq \delta/2, \label{eq_smallest_bla_2}
\end{equation}
but
\begin{equation}\label{eq_how_we_rescaled2}
\psi(\mathcal{D}_\mu \mcfK') \leq \delta/2\quad \textrm{ for } \mu>1.
\end{equation}

In particular, if $(\mcfK'',\mcfM'')$ is a tangent flow to $(\mcfK',\mcfM')$ at 0, then
\begin{equation}
\psi(\mcfK'')\leq \delta/2.
\end{equation}

We consider the following subcases:\\

Case 2a: $0$ is an interior point of $D'$.\\

In this case, it follows that
\begin{equation}
\phi(\mcfK'')\leq \delta/2,
\end{equation}
and thus $\mcfK''$ is a multiplicity 2 plane by Theorem \ref{thm:effective-sep} (isolation).\\

Case 2b: $0$ is a boundary point of $D' = \mathbb{H}$, and $\phi_h(\mcfK'')\leq \delta/2$.\\

In this case, it follows that $\phi_+(\mcfK'')\leq \delta/2$, so by Theorem \ref{thm:effective-sep} (isolation) $\mcfK''$ must be a multiplicity 2 halfplane.\\

Case 2c: $0$ is a boundary point of $D' = \mathbb{H}$, and $\phi(\mcfK'')\leq \delta/2$.\\

Then, by Theorem \ref{thm:effective-sep} (isolation), we have ${\phi}(\mcfK'')=0$ and $\mcfK''$ is the barrier plane. \\

In all Cases 2a, 2b and 2c, Theorem \ref{thm:quasi-tangent-flows} (tangent flows) implies that $(\mcfM',\mcfK')$ must be a limit flow, not a dilate of $(\mcfM,\mcfK)$. By applying Lemma \ref{thm:limit-minimal} (at different centres), it follows that there exists $b\in \mathbb{R}$ and $\Sigma$ containing the origin such that
$\mcfM'$ has (reflected) density at least 2 on $\Sigma\times (-\infty,b]$, where $\Sigma$ is given by either:
\begin{itemize}
\item a minimal hypersurface in $\mathbb{R}^{n+1}$ (Case 2a)
\item a free boundary minimal hypersurface in $\mathbb{H}$ (Case 2a or 2b)
\item the barrier plane $\partial \mathbb{H}$ (Case 2c)
\end{itemize}

In particular the (possibly reflected) flow satisfies:

\begin{equation}
\Theta(\widetilde{\mcfM'},\cdot)\geq 2\quad \textrm{on } \widetilde{\Sigma} \times (-\infty,b],
\end{equation}

where $\widetilde{\Sigma}$ is a minimal hypersurface in $\mathbb{R}^{n+1}$.

Again taking translates of $(\mcfM',\mcfK')$ respectively $(\widetilde{\mcfM'},\widetilde{\mcfK'})$ by $(x,t)\mapsto (x,t+j)$ and sending $j\to \infty$, we get a static limit flow $(\widehat{\mcfM},\widehat{\mcfK})$ defined in $\mathbb{R}^{n+1}$. In each Case 2a, 2b and 2c, (\ref{eq_smallest_bla_2}) prevents this flow from being planar, so by again monotonicity it must satisfy
\begin{equation}
\textrm{Ent}[\widehat{\mcfM}]>2.
\end{equation}

\bigskip

In all Cases 1 and 2, we have thus constructed a static (possibly reflected) limit flow $(\widehat{\mcfM},\widehat{\mcfK})$ defined in $\mathbb{R}^{n+1}$ that has entropy strictly larger than $2$. Let $(\mcfM^*,\mcfK^*)$ be a blowdown limit (i.e. a tangent flow at infinity) of $(\widehat{\mcfM},\widehat{\mcfK})$. Then $(\mcfM^*,\mcfK^*)$ is a static cone of multiplicity strictly larger than 2. Hence, by one-sided minimization it must be non-flat, contradicting the bound $|A|\leq CH$. This finishes the proof of the theorem.
\end{proof}

\bigskip

\section{Conclusion}

In this final section, we explain how to use Theorem \ref{thm_ell_reg} (elliptic regularization and consequences) and multiplicity one (see Section \ref{sec_mult_one}) to conclude the proofs of Theorem \ref{thm_size} (size of the singular set), Theorem \ref{thm_structure} (structure of singularities) and Theorem \ref{thm_longtime} (long-time behavior).

\subsection{Size of the singular set}

\begin{proof}[{Proof of Theorem \ref{thm_size}}]
By the local regularity theorem \cite{W05,Edelen17} a point $X$ is regular if and only if one sees a density one plane respectively halfplane as tangent flow at $X$.
By the results from Section \ref{sec_mult_one} (multiplicity one) we have several restrictions on the possible tangent flows. Namely, in case the barrier is mean-convex the tangent flows at singular points cannot be static or quasistatic, but must be selfsimilar shrinkers respectively selfsimilar shrinkers with free boundary.
For general barrier, the only additional case that can occur is quasistatic density two planes.
Hence, the assertion follows from dimension reduction, c.f. \cite[Sec. 9]{White_stratification}.
\end{proof}

\subsection{Structure of singularities}

\begin{proof}[{Proof of Theorem \ref{thm_structure}}]
Let $(\mathcal{M},\mathcal{K})$ be a mean-convex free boundary flow in $(D^{n+1},g)$, and assume that $D$ is mean-convex.

Given a point $X_0$ in the support of the flow we consider
\begin{multline}
\mathcal{F}_1:=\{ (\mathcal{M}',\mathcal{K}')\, |\, \textrm{$(\mathcal{M}',\mathcal{K}')$ is a limit flow at $X_0$,}\\
\textrm{which is defined in entire space} \},
\end{multline}
and
\begin{multline}
\mathcal{F}_2:=\{ (\widetilde{\mathcal{M}}',\widetilde{\mathcal{K}}')\, |\, \textrm{$(\mathcal{M}',\mathcal{K}')$ is a limit flow at $X_0$,}\\
\textrm{which is defined in a halfspace} \}.
\end{multline}
Let
\begin{equation}
\mathcal{F}:=\mathcal{F}_1\cup\mathcal{F}_2.
\end{equation}
As explained in the proof of Theorem \ref{thm_size} (size of the singular set), the class $\mathcal{F}$ does not contain any singular stationary cones. Moreover, whenever a tangent flow of some flow in the class $\mathcal{F}$ at some point is a static or quasistatic plane, then it is in fact a static multiplicity one plane. Hence,  the arguments from White's second paper \cite{White_structure} apply to the class $\mathcal{F}$ (actually with some simplifications thanks to the a priori bound $|A|\leq CH$ from Theorem \ref{thm_ell_reg}), yielding all assertions of Theorem \ref{thm_structure}, except for the last bullet point. In particular, note that the class $\mathcal{F}$ cannot contain any doubling of a (quasi)-static plane in a halfspace parallel to $\partial{\mathbb{H}}$, since otherwise it would also contain $\partial\mathbb{H}$.
Finally, the assertion in the last bullet point follows from \cite[Cor. 1.5]{HK_inscribed}.
 \end{proof}

\subsection{Long-time behavior}

\begin{proof}[{Proof of Theorem \ref{thm_longtime}}]
We first observe that the proof of the Bernstein-type theorem (Theorem \ref{thm:bern}) goes through if instead of the $|A|\leq CH$ bound (which unfortunately degenerates as $t\to\infty$) one uses the Schoen-Simon result about (mildly singular) minimal hypersurfaces in a slab \cite{SS} similarly as in \cite[Cor. 7.3]{White_size}. Consequently, the sheeting theorem (Theorem \ref{thm_sheeting}) can also be applied as $t\to\infty$.

Now, suppose $(\mathcal{M},\mathcal{K})$ in a mean-convex free boundary flow in a mean-convex domain $D$, such that
\begin{equation}
K_\infty:=\bigcap_t K_t \neq \emptyset.
\end{equation}
We argue as in \cite[Sec. 11]{White_size}.
Let $(\mathcal{M}^T,\mathcal{K}^T)$ be the result of translating $(\mathcal{M},\mathcal{K})$ by $(x,t)\mapsto (x,t-T)$. Then a subsequence will converge to a limit $(\mathcal{M}',\mathcal{K}')$. Note that
\begin{equation}\label{eq_long_limit}
\mathcal{K}'=\bigcap_t K_t\times \mathbb{R}.
\end{equation}
This is independent of the sequence of times going to infinity. Together with the fact that $\mathcal{M}'$ is determined by $\mathcal{K}'$ (the support of $\mathcal{M}'$ equals $\partial{\mathcal{K}}'$, and multiplicity one versus two is determined by whether or not the component of $\mathcal{K}'$ under consideration has interior points), we infer that
\begin{equation}
(\mathcal{M}^T,\mathcal{K}^T)\to (\mathcal{M}',\mathcal{K}')
\end{equation}
as $T\to\infty$ (i.e. it is not necessary to pass to a subsequence). By \eqref{eq_long_limit} the limit $(\mathcal{M}',\mathcal{K}')$ is static, i.e.
\begin{equation}
\textrm{spt}(\mathcal{M}')=\partial K_\infty \times \mathbb{R},\qquad K_t'=K_\infty.
\end{equation}
By one sided minimization (Theorem \ref{thm_limit_flows}) and the sheeting theorem (Theorem \ref{thm_sheeting}), $\partial K_\infty$ is a union of finitely many stable free boundary minimal surfaces. Hence, by Simons \cite{Simons} and Gr\"uter \cite{Gruter} the dimension of the singular set is at most $n-7$. This finishes the proof of Theorem \ref{thm_longtime}.
\end{proof}

\bigskip

\bibliographystyle{amsplain}

\vspace{10mm}

{\sc Nick Edelen, Department of Mathematics, University of Notre Dame, 255 Hurley Bldg, Notre Dame, IN 46556, USA}\\

{\sc Robert Haslhofer, Department of Mathematics, University of Toronto, 40 St George Street, Toronto, ON M5S 2E4, Canada}\\

{\sc Mohammad N. Ivaki, Department of Mathematics, University of Toronto, 40 St George Street, Toronto, ON M5S 2E4, Canada}\\

{\sc Jonathan J. Zhu, Department of Mathematics, Princeton University, Fine Hall, Washington Road, Princeton, NJ 08544, USA}\\

\noindent Email:  nedelen@nd.edu, roberth@math.toronto.edu,\\
 m.ivaki@utoronto.ca, jjzhu@math.princeton.edu

\end{document}